\documentclass{mcom-l}
\usepackage[figuresright]{rotating}
\usepackage{pdflscape}
\usepackage{amsfonts}
\usepackage{mathrsfs}
\usepackage{amsthm}
\usepackage{amsmath}
\usepackage{amssymb}
\usepackage{latexsym}
\usepackage{graphicx}
\usepackage{color,varioref}
\usepackage{longtable}
\usepackage{subfiles}
\usepackage{rotating,multirow,array}

\newtheorem{theorem}{Theorem}[section]

\newtheorem{proposition}[theorem]{Proposition}
\newtheorem{assumption}[theorem]{Assumption}

\theoremstyle{definition}

\theoremstyle{remark}
\newtheorem{remark}[theorem]{Remark}

\numberwithin{equation}{section}

\begin{document}

\title[An algorithm for the metric nearness problem]{An efficient algorithm for the $\ell_{p}$ norm based metric nearness problem}

\author[P.P. Tang]{Peipei Tang}
\address{School of Computer and
Computing Science, Zhejiang University City College, Hangzhou 310015, China}
\email{tangpp@zucc.edu.cn}

\author[B. Jiang]{Bo Jiang}
\address{School of Computer Science, Zhejiang University, No.38, Zheda Road, Hangzhou 310027, China}
\email{22021105@zju.edu.cn}

\author[C.J. Wang]{Chengjing Wang}
\address{School of Mathematics, Southwest Jiaotong University, No.999, Xian Road, West Park, High-tech Zone, Chengdu 611756, China}
\email{renascencewang@hotmail.com}
\thanks{{\bf The third author is the corresponding author}.}





\begin{abstract}
Given a dissimilarity matrix, the metric nearness problem is to find the nearest matrix of distances that satisfy the triangle inequalities. This problem has wide applications, such as sensor networks, image processing, and so on. But it is of great challenge even to obtain a moderately accurate solution due to the $O(n^{3})$ metric constraints and the nonsmooth objective function which is usually a weighted $\ell_{p}$ norm based distance. In this paper, we propose a delayed constraint generation method with each subproblem solved by the semismooth Newton based proximal augmented Lagrangian method (PALM) for the metric nearness problem. Due to the high memory requirement for the storage of the matrix related to the metric constraints, we take advantage of the special structure of the matrix and do not need to store the corresponding constraint matrix. A pleasing aspect of our algorithm is that we can solve these problems involving up to $10^{8}$ variables and $10^{13}$ constraints. Numerical experiments demonstrate the efficiency of our algorithm.

In theory, firstly, under a mild condition, we establish a primal-dual error bound condition which is very essential for the analysis of local convergence rate of PALM. Secondly, we prove the equivalence between the dual nondegeneracy condition and nonsingularity of the generalized Jacobian for the inner subproblem of PALM. Thirdly, when $q(\cdot)=\|\cdot\|_{1}$ or $\|\cdot\|_{\infty}$, without the strict complementarity condition, we also prove the equivalence between the the dual nondegeneracy condition and the uniqueness of the primal solution.
\end{abstract}

\maketitle
\section{Introduction}\label{sec:introduction}
In many fields, such as sensor networks, image processing, metric-based indexing of databases, computer vision and machine learning, quantities which measure the distances amongst points in a metric space should be given for further process (see e.g., \cite{Batra2008,Gentile2007,Vitaladevuni2010}). These distance measurements often need to satisfy the properties of a metric, especially the triangle inequality. However, the measurements may end up with a set of values that do not represent the actual distance values due to the noise of the data and errors of measurements, primarily because of the violation of the triangle inequality. Another related challenging problem is the one called correlation clustering (see e.g., \cite{Bansal2004}). Correlation clustering is an NP-hard problem  which aims to identify a weighed graph characterized by pairwise similarity and dissimilarity into groups and cluster the nodes in a way that minimizes the total quantity of mistakes. The mistake at the pair $(i,j)$ is $w_{ij}^{+}$ if the $i$-th and $j$-th nodes are separated but $w_{ij}^{-}$ if the $i$-th and $j$-th nodes are clustered together. Let $\mathcal{S}^{n}$ be the set of real symmetric matrices of order $n$. A matrix $M=(m_{ij})\in\mathcal{S}^{n}$ is called a distance matrix if it is a zero-diagonal matrix and satisfies the following triangle inequalities
\begin{eqnarray*}
	&&m_{ij}\leq m_{ik}+m_{jk},\quad \forall\, 1\leq i<j<k\leq n.
\end{eqnarray*}
Denote $\mathcal{M}_{n}$ be the set of distance matrices of order $n$. The correlation clustering problem can be written formally as the following integer linear programming problem
\begin{eqnarray*}
	&&\min_{X\in\mathcal{M}_{n}}\quad \sum_{1\leq i<j\leq n}\left(w_{ij}^{+}x_{ij}+w_{ij}^{-}(1-x_{ij})\right)\\
	&&\mbox{s.t.}\quad\quad x_{ij}\in\{0,1\},\ \forall\, 1\leq i<j\leq n.
\end{eqnarray*}

Input a set of distances, the metric nearness problem is to find the nearest set of distances satisfying the triangle inequalities in a high dimensional space. The metric nearness problem is a kind of matrix nearness problem (see e.g., \cite{Higham1989}) which aims to find a nearest member of some given class of matrices with some given property such as symmetry, positive definiteness, orthogonality, normality, rank-deficiency and instability. For inferring a metric, Xing et al. \cite{Xing2003} proposed a method for learning a Mahalanobis distance. Roth et al. \cite{Roth2003} applied the constant-shift embedding method (see e.g., \cite{Roth2002}) to metricize and subsequently construct a positive semidefinite matrix for denoising and clustering purposes. An alternative way to derive metric measures is to solve a metric multidimensional scaling problem (see e.g., \cite{Kruskal1978}) with its purpose of approximating the input distances by the distances between these points derived from a prescribed metric space, which is usually an Euclidean space.

In order to capture the confidence in individual dissimilarity measures, especially for the case of altogether missing distances, it is usually adopted as a strategy to use a weight matrix in the metric nearness problem. The approximation error which is called nearness is quantified by a weighted norm function that measures distortion between the input and output distances. There are many kinds of distortion functions including the vector $\ell_{p}$ norms $(p\geq 1)$ and Kullback-Leibler divergence. In this paper, we treat the strict upper triangular part of our matrices as vectors and consider the weighted $\ell_{p}$ norm as the distortion function. As shown in \cite{Veldt2019}, the aforementioned correlation clustering problem can be relaxed to an $\ell_{1}$ norm based metric nearness problem. Given a nonnegative symmetric zero-diagonal dissimilarity matrix $\widetilde{X}=(\tilde{x}_{ij})$, the $\ell_{p}$ norm based metric nearness problem takes the following form
\begin{eqnarray}\label{metric-nearness-problem}
	\min_{X\in\mathcal{M}_{n}} \left(\sum_{1\leq i<j\leq n}|w_{ij}(x_{ij}-\tilde{x}_{ij})|^{p}\right)^{1/p},
\end{eqnarray}
where $W=(w_{ij})$ is a weight matrix and its values reflect the relative confidence in the entries of the matrix $\widetilde{X}$.

During the last two decades, a lot of research has been done for the metric nearness problem. Based on the triangle inequality structure, Dhillon et al. \cite{Dhillon2005} presented a triangle fixing algorithm due to the inherent structure for efficiency gains. Several years later, the implementing details for the $\ell_{p}$ $(p=1,2,\infty)$  norm based metric nearness problem were presented by Brickell et al. \cite{Brickell2008}. They tried to remove the triangle inequality violations in order to improve the computational efficiency. Each iteration of the triangle fixing algorithm takes $O(n^{3})$ operations and it depends on a parameter which may influence the convergence of the algorithm and it is hard to obtain such a parameter with convergence guarantee. Veldt et al. \cite{Veldt2019} applied Dykstra's projection method \cite{Dykstra1983} to solve a regularized linear programming that is closely related to the original $\ell_{1}$ norm based metric nearness problem and also extended the algorithm to solve any metric constrained linear or quadratic programming problems. Due to the low memory requirement, the projection method can solve problems involving up to 11 thousand nodes, $6\times 10^7$ variables, and $7\times 10^{11}$ constraints. It is known from \cite{Escalante2011} that Dykstra's projection method is linearly convergent. A parallel projection method \cite{Ruggles2020} was derived to speed up the convergence in implementation. In order to find a faster scheme to remove the triangle inequality violations, Gabidolla et al. \cite{Gabidolla2019} used deep learning to remove the violations by minimally modifying the input distance matrix. Based on Floyd's shortest path algorithm and the Bregman projection, an active set algorithm was proposed in \cite{sonthalia2020project}, PROJECT AND FORGET, to solve the metric nearness problem.

Although the metric nearness problem (\ref{metric-nearness-problem}) is convex, it is of great challenge to solve it efficiently due to the nonsmooth objective function and $O(n^{3})$ involving constraints. As stated in \cite{Veldt2019}, the $\ell_{1}$ norm based metric nearness problem (\ref{metric-nearness-problem}) can be solved by an interior point method solver such as Gurobi and Mosek when $n$ is not too large. However, the memory cost of the interior point method may quickly become unacceptable when $n$ is large due to the large number of constraints ($n(n-1)(n-2)/2$ constraints). As we know, the parallel projection method \cite{Ruggles2020} is one of a few algorithms that successfully solve the $\ell_{1}$ norm based metric nearness problem (\ref{metric-nearness-problem}) with $n$ up to $10^{4}$, however, many iterations have to be taken with the computational cost $O(n^{3})$ for each iteration. As for the PROJECT AND FORGET algorithm, the memory requirements of the initial few iterations are usually very large for a large $n$.

For simplicity, let $n_{1}=n(n-1)/2$, $n_{2}=n(n-1)(n-2)/2$ and $\mbox{trivec}:\mathcal{S}^{n}\rightarrow\mathcal{R}^{n_{1}}$ be the vectorization operator defined by stacking the columns of the upper triangular part of an input matrix with $\mbox{trivec}(X)=[x_{12},x_{13},x_{23},\ldots,x_{1n},\ldots,\\x_{n-1,n}]^{T}$. By introducing an auxiliary vector $y=\mbox{trivec}(X-\widetilde{X})$, we can rewrite the metric nearness problem (\ref{metric-nearness-problem}) to the following convex composite optimization problem
\begin{eqnarray}\label{primal-problem}
	\min_{y\in\mathcal{R}^{n_{1}}}\Big\{h(Ay-b)+q(Dy)\Big\},
\end{eqnarray}
where $h(\cdot)=\delta_{C}(\cdot)$ and $C=\mathcal{R}_{-}^{n_{2}}$, $A$ is the constraint matrix corresponding to the triangle inequalities, $b=-A\mbox{trivec}(\widetilde{X})$, $D=\mbox{diag}(\mbox{trivec}(W))$ and $\delta_{C}$ is the indicator function of the set $C$, $q(\cdot)=\|\cdot\|_{p}$. All the diagonal entries of the weight matrix $D$ are positive. The dual problem related to problem (\ref{primal-problem}) takes the following form
\begin{eqnarray}\label{dual-problem}
	\min_{u,v}\Big\{h^{*}(u)+q^{*}(v)\, \Big|\, A^{T}u+D^{T}v=0\Big\}.
\end{eqnarray}

As for the convex composite problem (\ref{primal-problem}), there are many existing traditional algorithms available to obtain an approximate solution with a given accuracy. One of these algorithms is the augmented Lagrangian method (ALM) which dates back to \cite{Hestenes1969,Powell1969} and has been extensively studied for the general convex optimization problem in \cite{Rockafellar1976a}. It is also known from \cite{Rockafellar1976a} that for convex programming ALM applied to the primal problem is equivalent to the proximal point algorithm (PPA) applied to its dual form. Although ALM and its inexact variants such as the inexact PPAs proposed by \cite{Solodov1999,Solodov2000} possesses a fast local linear convergence property under some mild conditions, it is usually difficult to solve the corresponding inner subproblems exactly or to a high accuracy, especially for high dimensional composite nonsmooth problems. By introducing some slack variables, an alternative approach to solve problem (\ref{primal-problem}) is the alternating direction method of multipliers (ADMM) (see e.g., \cite{Gabay1976,Glowinski1975}) which deals with the corresponding variables alternately in each inner subproblem. One may refer to \cite{Glowinski2014} for the historical development of ADMM. The main challenge for these algorithms lies in two aspects. In one aspect, the memory requirement is huge when $n$ is large. Although there are only 3 nonzeros in each row of the constraint matrix, the storage needed for the matrix and the dual variable if necessary is still $O(n^{3})$, which may be unacceptable when $n$ is greater than $10^{4}$. In another aspect, the scale of each subproblem corresponding to those algorithms such as ALM and ADMM is $O(n^{2})$. The computational cost is very high, therefore how to find a highly accurate solution of each subproblem efficiently is a tricky problem we need to face. In theory, given the existence of primal and dual solutions, it is already known (see, e.g., \cite{Bonnans2000}) that the nondegeneracy condition of the primal problem implies a unique dual solution and the nondegeneracy condition of the dual problem implies a unique primal solution. However, the primal and dual nondegeneracy conditions do not imply the strict complementarity condition, which is necessary for the validation of the converses, except in the case of linear programming. Does this result still holds in some polyhedral setting such as problem (\ref{primal-problem}) with $p=1,\infty$?

In this paper, We take into account the special structure of the constraints and apply a delayed constraint generation method (DCGM) to deal with the $O(n^{3})$ constraints. We aim to design an asymptotically superlinearly convergent proximal augmented Lagrangian method (PALM) to solve each subproblem of DCGM. The inner subproblem of PALM is solved by the semismooth Newton (SsN) method, which fully takes advantage of the sparse structure of the corresponding Hessian matrix. DCGM, which is known as the cutting plane method \cite{Bertsimas1997}, is a famous approach to handle linear programming problems with a large number of constraints. During the implementation, the feasible set is approximated by a subset of the constraints and more constraints are added into the subset if the resulting solution is infeasible. Choosing a good initial subset of constraints also plays an important role in practical computation. Lin et al. \cite{LinSunToh2021} successfully applied DCGM to solve the shape constrained regression problems, and they randomly generate an initial subset of constraints. In our implementation, we use the zero vector as the initial iteration to generate an initial subset and also remove some constraints which may probably inactive, which reduces the computational cost greatly. Since the constraint matrix is very sparse with each row only 3 nonzero elements and not all of the variables are involved in each subset of constraints, especially for the first few iterations, we only need to solve an inner subproblem of PALM with the number of variables less than $n_{1}$. Numerically, as far as we know, it is the first time that numerical experiments on the $\ell_{\infty}$ norm based metric nearness problem are conducted in this paper with $n$ greater than $10^{4}$. In theory, firstly, under a mild condition, we establish a primal-dual error bound condition which is very essential for the analysis of local convergence rate of PALM. Secondly, we prove the equivalence between the dual nondegeneracy condition and the nonsingularity of the generalized Jacobian for the inner subproblem of PALM, which gives us a guidance of how to choose a proper proximal term of PALM. Thirdly, for the $\ell_{1},\ell_{\infty}$ norm based metric nearness problems, we prove that the nondegeneracy condition is equivalent to the strict Robinson constraint qualification (SRCQ) for the dual problem, which also implies that the critical cone of the primal problem contains only a zero element.

The remaining parts of this paper are organized as follows. In Section \ref{sec:preliminaries}, we introduce some basic concepts and preliminary results on variational analysis. In Section \ref{sec:dcg-palm}, the details of PALM based on DCGM are introduced with each inner subproblem solved by the SsN method. In Section \ref{sec:CQ-nondegeneracy-invertible}, the equivalence of the nondegeneracy condition of the dual problem and the nonsingularity of the corresponding Hessian matrix of the primal problem is obtained. In the case of $q(\cdot)=\|\cdot\|_{1}$ and $\|\cdot\|_{\infty}$, the equivalence between the nondegeneracy condition of the dual problem and the SRCQ is also given. In Section \ref{sec:computing-issues}, we present some computational issues related to our algorithm. In Section \ref{sec:Numerical issues}, we implement the numerical experiments to compare our algorithm with some existing algorithms to demonstrate the efficiency of the proposed algorithm. We conclude our paper in Section \ref{sec:Conclusion}.

\subsection{Additional notations}
\label{sec:Additional notations}
Let $\mathcal{R}^{n}$ be the Euclidean space with the standard inner product $\langle\cdot,\cdot\rangle$ and the norm $\|\cdot\|$. For a given self-adjoint positive semidefinite matrix $M\in\mathcal{R}^{n\times n}$, we denote $\lambda_{\min}(M)$ and $\lambda_{\max}(M)$ as the smallest and largest eigenvalues of $M$, respectively, and define $\|x\|_{M}:=\sqrt{\langle x,Mx\rangle}$ for any $x\in\mathcal{R}^{n}$. Given a set $C\subseteq\mathcal{R}^{n}$, $\mbox{lin}C$ denotes the lineality space with $\mbox{lin}C:=C\cap(-C)$, the relative interior of the set $C$ is denoted by $\mbox{ri}(C)$, the indicator function $\delta_{C}$ of the set $C$ is defined by $\delta_{C}(x)=0$ if $x\in C$, otherwise $\delta_{C}(x)=+\infty$. The weighted distance of $x$ to $C$ is defined by $\mbox{dist}_{M}(x,C):=\inf_{y\in C}\{\|y-x\|_{M}\}$. If $C=\emptyset$, we have $\mbox{dist}_{M}(x,C)=+\infty$ for all $x\in\mathcal{R}^{n}$. $I_{n}$ denotes the identity matrix of order $n$. If $M$ is an identity matrix, we just omit the subscript matrix $M$. For any cone $C\subseteq\mathcal{R}^{n}$, the polar of $C$ is defined to be the cone $C^{\circ}:=\{v\ |\ \langle v,w\rangle\leq0,\ \forall\ w\in C\}$.

\section{Preliminaries}\label{sec:preliminaries}
In this section, we introduce some basic preliminaries that will be used later.

Let $\mathbb{X}$, $\mathbb{Y}$  and $\mathbb{Z}$  be finite dimensional Hilbert spaces and $f:\mathbb{X}\rightarrow[-\infty,+\infty]$ be an extended real-valued function with its epigraph $\mbox{epi}f$ defined as the set
\begin{eqnarray*}
	\mbox{epi}f:=\Big\{(x,c)\ \Big|\ x\in\textrm{dom}(f),\ c\in\mathcal{R},\ f(x)\leq c\Big\}
\end{eqnarray*}
and its conjugate at $x\in\mathbb{X}$ defined by
\begin{eqnarray*}
	f^{*}(x):=\sup_{u\in\textrm{dom}(f)}\Big\{\langle x,u\rangle-f(u)\Big\}.
\end{eqnarray*}
The function $f$ is said to be proper if its epigraph is nonempty and contains no vertical lines, i.e., there exists at least one $x\in\mathbb{X}$ such that $f(x)<+\infty$ and $f(x)>-\infty$ for all $x\in\mathbb{X}$, or in other words, $f$ is proper if and only if $\textrm{dom}(f)$ is nonempty and $f$ is finite in its domain; otherwise, $f$ is improper. For a proper lower semicontinuous function $f$, recall the Moreau envelope function $e_{\sigma f}$ and the proximal mapping $\mbox{Prox}_{\sigma f}$ corresponding to $f$ with a parameter $\sigma>0$ as follows
\begin{eqnarray*}
	e_{\sigma f}(x)&:=&\inf_{u\in\mathbb{X}}\Big\{f(u)+\frac{1}{2\sigma}\|u-x\|^{2}\Big\},\\
	\mbox{Prox}_{\sigma f}(x)&:=&\mathop{\mbox{argmin}}_{u\in\mathbb{X}}\Big\{f(u)+\frac{1}{2\sigma}\|u-x\|^{2}\Big\}.
\end{eqnarray*}
Furthermore, if $f$ is convex, it follows from Theorem 2.26 of \cite{Rockafellar1998} that the Moreau envelope function $e_{\sigma f}$ is also convex and continuously differentiable with
\begin{eqnarray*}
	\nabla e_{\sigma f}(x)=\frac{1}{\sigma}\big(x-\mbox{Prox}_{\sigma f}(x)\big)
\end{eqnarray*}
and the proximal mapping $\mbox{Prox}_{\sigma f}$ is single-valued and continuous with the following Moreau identity (see e.g., Theorem 31.5 of \cite{Rockafellar1970}) holds
\begin{eqnarray*}
	\mbox{Prox}_{\sigma f}(x)+\sigma\mbox{Prox}_{f^{*}/\sigma}(x/\sigma)=x,\quad \forall\ x\in\mathbb{X}.
\end{eqnarray*}

A multifunction $\mathcal{F}:\mathbb{X}\rightrightarrows\mathbb{Y}$ is locally upper Lipschitz continuous at $x\in\mathbb{X}$ if there exist a parameter $\kappa$ which is independent of $x$ and a neighbourhood $\mathcal{U}$ of $x$ such that $\mathcal{F}(y)\subseteq\mathcal{F}(x)+\kappa\|y-x\|\mathbb{B}_{\mathbb{X}}$ holds for any $y\in\mathcal{U}$, where $\mathbb{B}_{\mathbb{X}}$ is a unit ball of the space $\mathbb{X}$. The multifunction $\mathcal{F}$ is said to be piecewise polyhedral if its graph $\mbox{gph}\mathcal{F}:=\{(x,y)\ |\ y\in\mathcal{F}(x)\}$ is the union of finitely many polyhedral convex sets. Furthermore, the inverse of a piecewise polyhedral multifunction is also piecewise polyhedral (see \cite{Robinson1981}). It has also been shown in \cite{Robinson1981} that a piecewise polyhedral multifunction is locally upper Lipschitz continuous everywhere.

The property of locally upper Lipschitz continuity has a fundamental relationship with the error bound condition. A multifunction $\mathcal{F}$ is said to satisfy the error bound condition if for any $\delta>0$ there exists $\kappa>0$ such that
\begin{eqnarray*}
	\mbox{dist}(x,\mathcal{F}^{-1}(0))\leq\kappa\mbox{dist}(0,\mathcal{F}(x))
\end{eqnarray*}
holds for any $x\in\{x\, |\, \mbox{dist}(x,\mathcal{F}^{-1}(0))\leq\delta\}$. It has been proven in Theorem 3H.3 of \cite{DontchevR2013} that a multifunction whose inverse is locally upper Lipschitz continuous at the origin satisfies the error bound condition.

Let $S$ be a subset of $\mathbb{X}$ and $x\in S$. The inner tangent cone and the contingent cone are defined as
\begin{eqnarray*}
	&&\mathcal{T}_{S}^{i}(x):=\liminf_{t\downarrow0}\frac{S-x}{t}
\end{eqnarray*}
and
\begin{eqnarray*}
	&&\mathcal{T}_{S}(x):=\limsup_{t\downarrow0}\frac{S-x}{t},
\end{eqnarray*}
respectively. If $S$ is convex and $x\in S$, it is known from Proposition 2.55 of \cite{Bonnans2000} that $\mathcal{T}^{i}_{S}(x)=\mathcal{T}_{S}(x)$ and it can be written equivalently as
\begin{eqnarray*}
	\mathcal{T}_{S}(x)=\Big\{d\in\mathbb{X}\, \Big|\, \mbox{dist}(x+td,S)=o(t),\ t\geq0\Big\}.
\end{eqnarray*}
The polar cone of the contingent cone is called the normal cone to $S$ at $x$, which can be defined as
\begin{eqnarray*}
	\mathcal{N}_{S}(x):=\left(\mathcal{T}_{S}(x)\right)^{\circ}.
\end{eqnarray*}
Therefore, if $S$ is convex and $x\in S$ we have
\begin{eqnarray*}
	\mathcal{N}_{S}(x)=\Big\{d\in\mathcal{X}\, \Big|\, \langle d,z-x\rangle\leq0,\ \forall\, z\in S\Big\}
\end{eqnarray*}
and $\mathcal{N}_{S}(x)=\emptyset$ if $x\notin S$.

Consider an extended real-valued function $f:\mathbb{X}\rightarrow\overline{\mathcal{R}}$ and a point $x\in\mathbb{X}$ such that $f$ is finite. The upper and lower directional derivatives of $f$ at $x$ are defined as
\begin{eqnarray*}
	&&f'_{+}(x;d):=\limsup_{t\downarrow0}\frac{f(x+td)-f(x)}{t}
\end{eqnarray*}
and
\begin{eqnarray*}
	&&f'_{-}(x;d):=\liminf_{t\downarrow0}\frac{f(x+td)-f(x)}{t},
\end{eqnarray*}
respectively. The function $f$ is said to be directionally differentiable at $x$ in the direction $d$ if $f'_{+}(x;d)=f'_{-}(x;d)$ and we denote it by $f'(x;d)$. The upper and lower directional epiderivatives of $f$ at $x$ are defined as
\begin{eqnarray*}
	&&f^{\downarrow}_{+}(x;d):=\sup_{t_{n}\in S}\left(\liminf_{{n\rightarrow\infty}\atop{d'\rightarrow d}}\frac{f(x+t_{n}d')-f(x)}{t_{n}}\right),
\end{eqnarray*}
and
\begin{eqnarray*}
	&&f^{\downarrow}_{-}(x;d):=\liminf_{{t\downarrow0}\atop{d'\rightarrow d}}\frac{f(x+td')-f(x)}{t},
\end{eqnarray*}
respectively, where $S$ is the set of all real positive sequences $\{t_{n}\}$ converging to zero. We say $f$ is directionally eipdifferentiable at $x$ in the direction $d$ if $f^{\downarrow}_{+}(x;d)=f^{\downarrow}_{-}(x;d)$ and we denote it by $f^{\downarrow}(x;d)$. Note that if $f$ is Lipschitz continuous near $x$, then $f'_{+}(x;d)=f^{\downarrow}_{+}(x;d)$ and $f'_{-}(x;d)=f^{\downarrow}_{-}(x;d)$ for all $d\in\mathbb{X}$. Furthermore, if $f$ is convex and $x\in\mathbb{X}$ with $f(x)$ finite, then the epiderivative $f^{\downarrow}(x;\cdot)$ exists and is convex.

Consider a convex composite problem which takes the following form
\begin{eqnarray}\label{composite-problem-p}
	\min_{x\in\mathbb{X}}\ \Big\{\hat{h}(\mathcal{A}x-\hat{b})+\hat{q}(\mathcal{W}x)\Big\},
\end{eqnarray}
where $\mathcal{A}:\mathbb{X}\rightarrow\mathbb{Y}$ and $\mathcal{W}:\mathbb{X}\rightarrow\mathbb{Z}$ are linear mappings, $\hat{b}\in\mathbb{Y}$ is a given data, $\hat{h}:\mathcal{Y}\rightarrow(-\infty,+\infty]$ and $\hat{q}:\mathcal{W}\rightarrow(-\infty,+\infty]$ are two closed proper convex functions. Let $F(x)=(\mathcal{A}x-b,\mathcal{W}x)$ and $\hat{g}(y,z)=\hat{h}(y)+\hat{q}(z)$, problem (\ref{composite-problem-p}) can be reformulated as the following form
\begin{eqnarray}\label{composite-problem}
	\min_{x\in\mathbb{X}}\ \hat{g}(F(x)).
\end{eqnarray}
The above problem (\ref{composite-problem}) is equivalent to
\begin{eqnarray}\label{composite-problem-equivalent}
	\min_{(x,c)\in\mathbb{X}\times\mathcal{R}}\ \Big\{c\ \Big|\ (F(x),c)\in\mbox{epi}\hat{g}\Big\}.
\end{eqnarray}
The Lagrangian function of (\ref{composite-problem-equivalent}) is
\begin{eqnarray*}
	\mathcal{L}((x,c),(w,\gamma))=\langle w,F(x)\rangle+c(1+\gamma).
\end{eqnarray*}
Let $\mathcal{K}:=\mbox{epi}\hat{g}$. The Karush-Kuhn-Tucker (KKT) system takes the following form
\begin{eqnarray*}
	\nabla F(x)w=0,\ \gamma = -1,\ (w,\gamma)\in\mathcal{N}_{\mathcal{K}}(F(x),\hat{g}(F(x)).
\end{eqnarray*}
Let $(\bar{x},\bar{c})$ be an optimal solution of problem (\ref{composite-problem-equivalent}) with $(\bar{w},-1)$ as the corresponding dual solution. Since the Slater condition always holds for the function $f(x,c):=\hat{g}(F(x))-c$, it is known from Proposition 2.61 of \cite{Bonnans2000} that the tangent cone to $\mathcal{K}$ at the point $(F(\bar{x}),\bar{c})$ is given by
\begin{eqnarray}
	\mathcal{T}_{\mathcal{K}}(F(\bar{x}),\bar{c})=\Big\{(dw,dc)\ \Big|\ \hat{g}^{\downarrow}_{-}(F(\bar{x});dw)\leq dc\Big\}.\label{tangent-cone}
\end{eqnarray}
The Robinson constraint qualification (RCQ) is said to hold at $(\bar{x},\bar{c})$ of problem (\ref{composite-problem-equivalent}) if
\begin{eqnarray}\label{RCQ-composite-problem-equivalent}
	\left(F'(\bar{x}),1\right)(\mathbb{X}\times\mathcal{R})+\mathcal{T}_{\mathcal{K}}(F(\bar{x}),\bar{c})=\mathbb{X}\times\mathcal{R}.
\end{eqnarray}
It follows from Theorem 3.9 and Proposition 3.17 of \cite{Bonnans2000} that the RCQ (\ref{RCQ-composite-problem-equivalent}) holds at a solution $(\bar{x},\bar{c})$ if and only if its related dual solution set is a nonempty, convex and compact set. The SRCQ is said to hold at $(\bar{x},\bar{c})$ for $(\bar{w},-1)$ if
\begin{eqnarray}\label{SRCQ-composite-problem-equivalent}
	\left(F'(\bar{x}),1\right)(\mathbb{X}\times\mathcal{R})+\mathcal{T}_{\mathcal{K}}(F(\bar{x}),\bar{c})\cap (\bar{w},-1)^{\perp}=\mathbb{X}\times\mathcal{R}.
\end{eqnarray}
If the SRCQ (\ref{SRCQ-composite-problem-equivalent}) holds at $(\bar{x},\bar{c})$, then the corresponding dual optimal solution set is a singleton (see e.g., Proposition 4.50 of \cite{Bonnans2000}).

By \eqref{tangent-cone}, we have
\begin{eqnarray}\label{composite-problem-tangent-cone-1}
	\mathcal{T}_{\mathcal{K}}(F(\bar{x}),\bar{c})\cap (\bar{w},-1)^{\perp}&=&\left\{(dw,dc)\ |\ \hat{g}^{\downarrow}_{-}(F(\bar{x});dw)\leq dc,\ \langle \bar{w},dw\rangle-dc=0\right\}\nonumber\\
	&=&\left\{(dw,dc)\ |\ \hat{g}^{\downarrow}_{-}(F(\bar{x});dw)\leq \langle \bar{w},dw\rangle,\ \langle \bar{w},dw\rangle-dc=0\right\}.\nonumber\\
&&
\end{eqnarray}
Note from Corollary 2.4.9 of \cite{Clarke1990} that
\begin{eqnarray}\label{eq:equivalent-normal-partial}
	(\bar{w},-1)\in\mathcal{N}_{\mathcal{K}}(F(\bar{x}),\bar{c})\quad\Leftrightarrow\quad \bar{w}\in\partial \hat{g}(F(\bar{x})),
\end{eqnarray}
then we have
\begin{eqnarray}\label{composite-problem-tangent-cone-2}
	\hat{g}^{\downarrow}_{-}(F(\bar{x});d)\geq \langle \bar{w},d\rangle,\quad \forall\, d\in\mathbb{Y}.
\end{eqnarray}
By combing (\ref{composite-problem-tangent-cone-1}) and (\ref{composite-problem-tangent-cone-2}), we conclude that
\begin{eqnarray*}
	\mathcal{T}_{\mathcal{K}}(F(\bar{x}),\bar{c})\cap (\bar{w},-1)^{\perp}&=&\left\{(dw,dc)\ \Big|\ \hat{g}^{\downarrow}_{-}(F(\bar{x});dw)= \langle\bar{w},dw\rangle=dc\right\}
\end{eqnarray*}
and the SRCQ of problem (\ref{composite-problem}) at $\bar{x}$ for $\bar{w}$ takes the following form
\begin{eqnarray*}
	F'(\bar{x})\mathbb{X}+\Big\{dw\ \Big|\ \hat{g}^{\downarrow}_{-}(F(\bar{x});dw)=\langle \bar{w},dw\rangle\Big\}=\mathbb{Y}.
\end{eqnarray*}

The nondegeneracy condition of problem (\ref{composite-problem-equivalent}) is said to hold at $(\bar{x},\bar{c})$ if
\begin{eqnarray*}
	(F'(\bar{x}),1)(\mathbb{X}\times\mathcal{R})+\mbox{lin}(\mathcal{T}_{\mathcal{K}}(F(\bar{x}),\bar{c}))=\mathbb{Y}.
\end{eqnarray*}
Due to the formulation of $\mathcal{T}_{\mathcal{K}}(F(\bar{x}),\bar{c})$, the nondegeneracy condition of problem (\ref{composite-problem}) at $\bar{x}$ takes the following form
\begin{eqnarray*}
	F'(\bar{x})\mathbb{X}+\Big\{d\ \Big|\ \hat{g}^{\downarrow}_{-}(F(\bar{x});d)\leq-\hat{g}^{\downarrow}_{-}(F(\bar{x});-d)\Big\}=\mathbb{Y}.
\end{eqnarray*}
The critical cone related to problem (\ref{composite-problem-equivalent}) at $(\bar{x},\bar{c})$ takes the following form
\begin{eqnarray*}
	\mathcal{C}(\bar{x},\bar{c})=\Big\{(dx,dc)\ \Big|\ \hat{g}^{\downarrow}_{-}\big(F(\bar{x}),F'(\bar{x}) dx\big)=0,\ dc=0\Big\},
\end{eqnarray*}
which can be written equivalently as
\begin{eqnarray}\label{critical-cone-equivalent}
	\mathcal{C}(\bar{x},\bar{c})=\Big\{(dx,dc)\ \Big|\ (F'(\bar{x})dx,dc)\in\mathcal{T}_{\mathcal{K}}(F(\bar{x}),\bar{c})\ \cap\ \{(\bar{w},-1)\}^{\perp}\Big\}.
\end{eqnarray}
The critical cone of problem (\ref{composite-problem}) at $\bar{x}$ can be written in the form
\begin{eqnarray*}
	\mathcal{C}(\bar{x})=\Big\{dx\ \Big|\ \hat{g}^{\downarrow}_{-}\big(F(\bar{x}),F'(\bar{x}) dx\big)=0\Big\}.
\end{eqnarray*}

\begin{proposition}\label{prop-critical-cone-linearspace-dualvalue}
	The critical cone $\mathcal{C}(\bar{x})$ is a linear subspace if and only if
	\begin{eqnarray*}
		\bar{w}\in\mbox{ri}(\partial\hat{g}(F(\bar{x}))).
	\end{eqnarray*}
\end{proposition}
\begin{proof}
	The critical cone $\mathcal{C}(\bar{x})$ is a linear subspace if and only if the critical cone $\mathcal{C}(\bar{x},\bar{c})$ is a linear subspace. Combining the equivalent form (\ref{critical-cone-equivalent}) of the critical cone $\mathcal{C}(\bar{x},\bar{c})$ with $(\bar{w},-1)\in\mathcal{N}_{\mathcal{K}}(F(\bar{x}),\bar{c})$, the critical cone $\mathcal{C}(\bar{x},\bar{c})$ is a linear subspace if and only if
	\begin{eqnarray*}
		(\bar{w},-1)\in\mbox{ri}\left(\mathcal{N}_{\mathcal{K}}(F(\bar{x}),\bar{c})\right),
	\end{eqnarray*}
	due to Proposition 2.4.1 of \cite{FacchineiPand2003}.
	The desired result follows from \eqref{eq:equivalent-normal-partial} and Theorem 6.8 of \cite{Rockafellar1970}.
\end{proof}

\begin{assumption}\label{assumption}
	Assume that zero is not an optimal solution of problem (\ref{primal-problem}). 
\end{assumption}
Assumption \ref{assumption} makes sense due to the basic fact that zero is the unique solution of problem (\ref{primal-problem}) if all the elements of $b$ are nonnegative.

Now we state some basic results for further use. First, we present a known proposition as below, which can be seen in Theorem 6.46 of \cite{Rockafellar1998}.
\begin{proposition}\label{prop-Tangent-Normal-Cone}
	For a polyhedral set $C=\{x\, |\, \hat{A}x-\hat{b}\leq0\}$, where $\hat{A}\in \mathcal{R}^{\hat{m}\times \hat{n}}$ and $\hat{b}\in \mathcal{R}^{\hat{m}}$, we have
	\begin{eqnarray*}
		\mathcal{T}_{C}(x)&=&\Big\{d\, \Big| \, \langle \hat{A}_{i}^{T},d\rangle\leq0,\ \mbox{for}\ i\in \mathcal{I}(x)\Big\},\\
		\mathcal{N}_{C}(x)&=&\Big\{\hat{A}^{T}d\ \Big|\ d_{i}\geq0,\ \mbox{for}\ i\in\mathcal{I}(x),\ d_{i}=0,\ \mbox{for}\ i\notin\mathcal{I}(x)\Big\},\\
	\end{eqnarray*}
	where $\mathcal{I}(x)=\Big\{i\, \Big|\, \langle \hat{A}_{i}^{T},x\rangle-\hat{b}_{i}=0\Big\}$, $\hat{A}_{i}$ is the $i$th row of the matrix $\hat{A}$.
\end{proposition}
In the following, we introduce some facts related to the functions $h$ and $q$. It is easy to compute the conjugate function of $h$ as
\begin{eqnarray*}
	h^{*}(u) = \delta_{C_{1}}(u),
\end{eqnarray*}
where $C_{1}=\mathcal{R}^{n_{2}}_{+}$. Then $\mbox{Prox}_{h^{*}}(u)=\Pi_{C_{1}}(u)$.
Based on Example 2.67 of \cite{Bonnans2000}, for $d\in\mathcal{R}^{n_{2}}$ we have
\begin{eqnarray*}
	h^{*\downarrow}_{-}(u;d)=\delta_{\mathcal{T}_{C_{1}}(u)}(d)
\end{eqnarray*}
with
\begin{eqnarray*}
	\mathcal{T}_{C_{1}}(u)&=&\Big\{d\ \Big|\ d_{i}\geq0,\ \mbox{if}\ u_{i}=0\Big\},\\
	\mbox{lin}\mathcal{T}_{C_{1}}(u)&=&\Big\{d\ \Big|\ d_{i}=0,\ \mbox{if}\ u_{i}=0\Big\},
\end{eqnarray*}
due to Proposition \ref{prop-Tangent-Normal-Cone}. Given $\sigma>0$, some basic results related to the function $q$ are presented in the following form.
\begin{enumerate}
	\item $q(\cdot)=\|\cdot\|_{1}$ or $q(\cdot)=\|\cdot\|_{\infty}$.
	
	There exists $C_{2}=\{v\, |\, \bar{A}v\leq \bar{b}\}$ with $\bar{A}\in\mathcal{R}^{m_{1}\times n_{1}}$ and $\bar{b}\in\mathcal{R}^{m_{1}}$ (the dependence of $\bar{A},\bar{b}$ with $p$ is dropped out here) such that for any $v\in\mathcal{R}^{n_{1}}$, we have
	\begin{eqnarray*}
		q^{*}(v)&=&\delta_{C_{2}}(v),\\
		\mbox{Prox}_{\sigma q^{*}}(v)&=&\Pi_{C_{2}}(v).
	\end{eqnarray*}
	For more details, if $q(\cdot)=\|\cdot\|_{1}$, we have
	\begin{eqnarray*}
		\left(\Pi_{C_{2}}(v)\right)_{i}=\left\{\begin{array}{ll}v_{i},&\mbox{if}\ |v_{i}|\leq1,\\\mbox{sign}(v_{i}),&\mbox{otherwise}.\end{array}\right.
	\end{eqnarray*}
	If $q(\cdot)=\|\cdot\|_{\infty}$, we have
	\begin{eqnarray*}
		\Pi_{C_{2}}(v)=\left\{\begin{array}{ll}v,&\mbox{if}\ \|v\|_{1}\leq 1,\\
			P_{v}\Pi_{\triangle}(P_{v}v),&\mbox{otherwise},\end{array}\right.
	\end{eqnarray*}
	where $P_{v}=\mbox{diag}(\mbox{sign}(v))$, $\triangle=\{v\, |\, v_{1}+\ldots+v_{n_{1}}=1\}$ and $\Pi_{\triangle}(P_{v}v)$ can be computed in $O(n_{1}\log n_{1})$ operations. One may see Appendix A of \cite{LinSunToh2021} for more details.
	
	It is known from Proposition \ref{prop-Tangent-Normal-Cone} that
	\begin{eqnarray*}
		\mathcal{T}_{C_{2}}(v)=\Big\{d\ \Big|\ \langle \bar{A}_{j}^{T},d\rangle\leq0,\ j\in J_{2}\Big\}
	\end{eqnarray*}
	with $J_{2}:=\{j\ |\ \bar{A}_{j}v=\bar{b}_{j}\}$.
	Hence we have
	\begin{eqnarray*}
		\mbox{lin}\mathcal{T}_{C_{2}}(v)&=&\cap_{j\in J_{2}}\{\bar{A}_{j}^{T}\}^{\perp}.
	\end{eqnarray*}
	\item $q(\cdot)=\|\cdot\|$.
	
	Let $C_{2}=\{v\, |\, \|v\|\leq 1\}$. We have
	\begin{eqnarray*}
		q^{*}(v)&=&\delta_{C_{2}}(v),\\
		\mbox{Prox}_{q^{*}}(v)&=&\Pi_{C_{2}}(v)=\left\{\begin{array}{ll}v,&\mbox{if}\ \|v\|\leq1,\\
			\frac{v}{\|v\|},&\mbox{otherwise},\end{array}\right.
	\end{eqnarray*}
	and
	\begin{eqnarray*}
		&&\mathcal{T}_{C_{2}}(v)=\left\{\begin{array}{cc}\{d\, |\, \langle d,v\rangle\leq 0\},&\quad\mbox{if}\quad\|v\|=1,\\
			\mathcal{R}^{n_{1}},&\quad\mbox{if}\quad\|v\|<1,\end{array}\right.\\
		&&\mbox{lin}\mathcal{T}_{C_{2}}(v)=\left\{\begin{array}{cc}\{v\}^{\perp},&\quad\mbox{if}\quad\|v\|=1,\\
			\mathcal{R}^{n_{1}},&\quad\mbox{if}\quad\|v\|<1.\end{array}\right.
	\end{eqnarray*}
\end{enumerate}
\section{A DCGM based PALM for the primal problem}\label{sec:dcg-palm}
In this section, we introduce PALM for the primal problem (\ref{primal-problem}) with DCGM to improve the efficiency. PALM is an efficient algorithm to solve the metric nearness problem due to its superlinear convergence which will be proved later. Although we can apply PALM directly to the original primal problem (\ref{primal-problem}), it is of great challenge to obtain a desired approximate solution of the corresponding large scale subproblem when $n$ is huge. DCGM approximates the feasible set by only a subset of the constraints with more constraints added if the resulting solution is infeasible. We apply PALM to each reduced problem generated by DCGM, which reduces the computational cost greatly. Taking advantage of the special structure of these problems, an SsN method is used to find an approximate solution for each subproblem of PALM. An interesting thing we need to mention is that we do not need to store the corresponding constraint matrix. The implementation details will be introduced in the following sections.

\subsection{PALM}
Before introducing PALM, we first write out the Lagrangian function for the composite problem (\ref{primal-problem}).
\begin{eqnarray*}
	l(y;u,v)&=&\inf_{s\in\mathcal{R}^{n_{2}},t\in\mathcal{R}^{n_{1}}}\Big\{h(Ay-b-s)+q(Dy-t)+\langle u,s\rangle+\langle v,t\rangle\Big\}\\
	&=&-h^{*}(u)-q^{*}(v)+\langle u,Ay-b\rangle+\langle v,Dy\rangle.
\end{eqnarray*}
The KKT condition for the composite problem takes the following form
\begin{eqnarray}\label{composite-problem-kkt}
	A^{T}u+D^{T}v=0,\ Ay-b\in\partial h^{*}(u),\ Dy\in\partial q^{*}(v).
\end{eqnarray}
Given $\sigma>0$, the augmented Lagrangian function for the composite problem (\ref{primal-problem}) is
\begin{eqnarray*}
	\mathcal{L}_{\sigma}(y;u,v)&=&\sup_{s\in\mathcal{R}^{n_{2}},t\in\mathcal{R}^{n_{1}}}\Big\{l(y,s,t)-\frac{1}{2\sigma}\|s-u\|^{2}-\frac{1}{2\sigma}\|t-v\|^{2}\Big\}\\&=&\sigma^{-1}e_{\sigma h^{*}}(u+\sigma(Ay-b))+\sigma^{-1}e_{\sigma q^{*}}(v+\sigma Dy)\\
	&&+\langle u,Ay-b\rangle+\frac{\sigma}{2}\|Ay-b\|^{2}+\langle v,Dy\rangle+\frac{\sigma}{2}\|Dy\|^{2}.
\end{eqnarray*}
Now we are ready to introduce the following PALM.
\bigskip

\noindent\fbox{\parbox{\textwidth}{\noindent{\bf Algorithm 1 (PALM): \label{palm}} Let $\{\nu_{k}\}$ be a nonnegative summable sequence and $\{H_{k}\}$ be a sequence of positive definite matrix satisfying $H_{k+1}\preceq(1+\nu_{k})H_{k}$, $H_{k}\succeq\lambda_{\min}I_{n_{1}}$ for $k\geq0$ and $\limsup_{k\rightarrow\infty}\lambda_{\max}(H_{k})=\lambda_{\infty}$ with $0<\lambda_{\min}\leq\lambda_{\infty}<+\infty$. Given $\sigma_{0}>0$, choose $y^{0}\in\mathcal{R}^{n_{1}},u^{0}\in\mathcal{R}^{n_{2}},v^{0}\in\mathcal{R}^{n_{1}}$. Set $k=0$ and iterate:
		\begin{description}
			\item [Step 1.] Find an approximate solution
			\begin{eqnarray}
				y^{k+1}&\approx&\mathop{\rm argmin}_{y\in\mathcal{R}^{n_{1}}}\Big\{\phi_{k}(y):=\mathcal{L}_{\sigma_{k}}(y;u^{k},v^{k})+\frac{1}{2\sigma_{k}}\|y-y^{k}\|_{H_{k}}^{2}\Big\}.\label{eq:PALM}
			\end{eqnarray}
			\item [Step 2.] Update
			\begin{eqnarray*}
				u^{k+1}&=&u^{k}+\sigma_{k}(Ay^{k+1}-b-\mbox{Prox}_{\sigma_{k}^{-1}h}(Ay^{k+1}-b+\sigma_{k}^{-1}u^{k}))\\
				&=&\mbox{Prox}_{\sigma_{k}h^{*}}(\sigma_{k}(Ay^{k+1}-b)+u^{k}),\\
				v^{k+1}&=&v^{k}+\sigma_{k}(Dy^{k+1}-\mbox{Prox}_{\sigma_{k}^{-1}q}(Dy^{k+1}+\sigma_{k}^{-1}v^{k}))\\
				&=&\mbox{Prox}_{\sigma_{k}q^{*}}(\sigma_{k}Dy^{k+1}+v^{k}).
			\end{eqnarray*}
			\item [Step 3.] If a desired stopping criterion is satisfied, terminate; otherwise, update $\sigma_{k+1}\uparrow\sigma_{\infty}\leq\infty$ and return to Step 1.
\end{description}}}

\bigskip

\subsection{The convergence results of PALM}
In this subsection, we introduce the convergence results of PALM mentioned above. One may see \cite{LiSunToh2020} for more details. Let $M_{k}:=\mbox{Diag}(H_{k},I,I)$, $\Omega$ be the solution set of the KKT system (\ref{composite-problem-kkt}) and $\mathcal{T}_{l}$ be a maximal monotone operator with
\begin{eqnarray*}
	\mathcal{T}_{l}(y,u,v):=\Big\{(y',u',v')\ \Big|\ (y',-u',-v')\in\partial l(y,u,v)\Big\}.
\end{eqnarray*}
Let $\{\epsilon_{k}\}$ and $\{\delta_{k}\}$ be two given summable sequences with $\epsilon_{k}\geq0$ and $0\leq\delta_{k}<1$ for all $k\geq0$. There are two general stopping criteria for the inner subproblem (\ref{eq:PALM}).
\begin{eqnarray*}
	(A')\quad\|\nabla\phi_{k}(y^{k+1})\|&\leq&\frac{\epsilon_{k}\sqrt{\lambda^{k}_{\min}}}{\sigma_{k}},\\
	(B')\quad\|\nabla\phi_{k}(y^{k+1})\|&\leq&\frac{\delta_{k}\sqrt{\lambda^{k}_{\min}}}{\sigma_{k}}\|(y^{k+1},u^{k+1},v^{k+1})-(y^{k},u^{k},v^{k})\|_{M_{k}},
\end{eqnarray*}
where $\lambda^{k}_{\min}=\min\big\{\lambda_{\min}(H_{k}),1\big\}$. Now we present the convergence results of PALM.
\begin{theorem}
	Let $\{(y^{k},u^{k},v^{k})\}$ be the sequence generated by PALM with the stopping criterion $(A')$. Then the sequence $\{(y^{k},u^{k},v^{k})\}$ is bounded with $\{y^{k}\}$ converging to an optimal solution of the primal problem (\ref{primal-problem}) and $\{(u^{k},v^{k})\}$ converging to an optimal solution of the dual problem (\ref{dual-problem}).
\end{theorem}

\begin{theorem}\label{thm:palm-convergence-rate}
Let $\delta$ be a real number with $\delta>\sum_{k=0}^{\infty}\epsilon_{k}$. Assume that there exists $\kappa>0$ such that for any $(y,u,v)$ satisfying $\mbox{dist}((y,u,v),\mathcal{T}_{l}^{-1}(0))<\delta$ it holds that
	\begin{eqnarray}\label{cond:error-bound}
		\mbox{dist}((y,u,v),\mathcal{T}_{l}^{-1}(0))\leq\kappa\mbox{dist}(0,\mathcal{T}_{l}(y,u,v)).
	\end{eqnarray}
	Let $\{(y^{0},u^{0},v^{0})\}$ be an initial point with $\mbox{dist}_{M_{0}}((y^{0},u^{0},v^{0}),\mathcal{T}_{l}^{-1}(0))<\delta-\sum_{k=0}^{\infty}\epsilon_{k}$. Then the sequence $\{(y^{k},u^{k},v^{k})\}$ generated by PALM under the criteria $(A')$ and $(B')$ has the following property
	\begin{eqnarray*}
		\mbox{dist}_{M_{k+1}}((y^{k+1},u^{k+1},v^{k+1}),\mathcal{T}_{l}^{-1}(0))\leq\mu_{k}\mbox{dist}_{M_{k}}((y^{k},u^{k},v^{k}),\mathcal{T}_{l}^{-1}(0)),\,  \forall\, k\geq0,
	\end{eqnarray*}
	where
	\begin{eqnarray*}
		\mu_{k}=\frac{1+\nu_{k}}{1-\delta_{k}}\left(\delta_{k}+\frac{(1+\delta_{k})\kappa\lambda_{m}^{k}}{\sqrt{\sigma_{k}^{2}+\kappa^{2}(\lambda_{m}^{k})^{2}}}\right)\rightarrow\mu_{\infty}=\frac{\kappa\lambda_{m}}{\sqrt{\sigma_{\infty}^{2}+\kappa^{2}\lambda_{m}^{2}}}<1,
	\end{eqnarray*}
	$\lambda_{m}^{k}=\max\{\lambda_{max}(H_{k}),1\}$ and $\lambda_{m}=\max\{\lambda_{max}(H_{\infty}),1\}$.
\end{theorem}

\begin{remark}\label{rem:palm}
	It is interesting that if $H_{k}\equiv I_{n_{1}}$, then the above PALM becomes the classical PALM which dates back to \cite{Rockafellar1976a,Rockafellar1976b}. The proximal term $\frac{1}{2\sigma_{k}}\|y-y^{k}\|^{2}_{H_{k}}$ added in each subproblem not only guarantees the nonsingularity of the corresponding Hessian matrix but also improves the efficiency of PALM. However, the convergence rate of PALM also depends on $\{H_{k}\}$ and we need to balance these two counterparts. It is usually difficult to choose an appropriate proximal term in advance. Fortunately, we prove an equivalent condition in Section \ref{sec:CQ-nondegeneracy-invertible} to measure the nonsingularity of the corresponding Hessian matrix without the proximal term $\frac{1}{2\sigma_{k}}\|y-y^{k}\|^{2}_{H_{k}}$. Therefore, we can update $H_{k}$ adaptively, which helps to improve the performance of the algorithm.
\end{remark}

From Theorem \ref{thm:palm-convergence-rate}, an error bound condition for the maximal monotone $\mathcal{T}_{l}$ is needed for the local convergence rate of PALM. As mentioned in Section \ref{sec:preliminaries}, every polyhedral multifunction is upper Lipschitz continuous at every point of its domain, therefore it also satisfies the error bound condition (\ref{cond:error-bound}). Hence, for the case of $p=1,\infty$, the error bound condition for the maximal monotone $\mathcal{T}_{l}$ automatically holds. As for the case of $p=2$, we can prove the corresponding error bound condition based on a similar proof idea from \cite{IzmailovKS} as follows.

\begin{proposition}
	For the $\ell_{2}$ norm based metric nearness problem (\ref{primal-problem}), let $(\bar{y},\bar{u},\bar{v})$ be a solution of the KKT system (\ref{composite-problem-kkt}). Suppose that Assumption \ref{assumption} holds. Then the desired error bound condition (\ref{cond:error-bound}) is valid at $(\bar{y},\bar{u},\bar{v})$.
\end{proposition}
\begin{proof}
	Let $\mathcal{M}(\bar{y})$ be the set of dual
	solutions associated with $\bar{y}$. In
	order to prove the error bound condition of $\mathcal{T}_{l}$ at
	$(\bar{y},\bar{u},\bar{v})$, we first claim that
	there exists a neighborhood $\mathcal{U}$ of
	$(\bar{y},\bar{u},\bar{v})$ such that for any
	solution $(y(t),u(t),v(t))\in\mathcal{U}$ of the perturbed
	KKT system
	\begin{eqnarray*}
		A^{T}u+D^{T}v-t_{1}=0,\quad Ay-b-t_{2}\in\partial h^{*}(u),\quad Dy-t_{3}\in\partial q^{*}(v),
	\end{eqnarray*}
	it satisfies the following estimation
	\begin{eqnarray}\label{eq:estimation}
		\|y(t)-\bar{y}\|=O(\|t\|),
	\end{eqnarray}
	where the norm of $t=(t_{1},t_{2},t_{3})$ is sufficiently
	small. We first suppose that the claim is not true. Therefore we can
	find some sequences
	$\{t^{k}:=(t_{1}^{k},t_{2}^{k},t_{3}^{k})\}$ and
	$\{(y^{k},u^{k},v^{k})\}$ such that $t^{k}\rightarrow
	0$,
	$(y^{k},u^{k},v^{k})\rightarrow(\bar{y},\bar{u},\bar{v})$,
	$(y^{k},u^{k},v^{k})$ is a solution of the perturbed KKT
	system for $t=t^{k}$ and
	\begin{eqnarray}
		\|y^{k}-\bar{y}\|>\gamma_{k}\|t^{k}\|
		\label{ineq:gammak}
	\end{eqnarray}
	for some $\gamma_{k}>0$ with $\gamma_{k}\rightarrow\infty$. In the
	following proof, we may pass to the subsequence if necessary. From
	the above assumption, we can see that
	$\left\{\frac{y^{k}-\bar{y}}{\|y^{k}-\bar{y}\|}\right\}$
	converges to some $d$ with $\|d\|=1$. Let
	$s_{k}=\|y^{k}-\bar{y}\|$. Since $h^{*}(u)=\delta_{C}^{*}(u)$ with $C=\mathcal{R}^{n_{2}}_{-}$, for each $k$ by the
	perturbed KKT system we have,
	\begin{eqnarray*}
		Ay^{k}-b-t_{2}^{k}=\Pi_{C}(Ay^{k}-b-t_{2}^{k}+u^{k}).
	\end{eqnarray*}
	It follows from Theorem 4.1.1 of \cite{FacchineiPand2003} that
	\begin{eqnarray}
		Ay^{k}-b-t_{2}^{k}&=&\Pi_{C}(A\bar{y}-b+\bar{u})+\Pi_{\mathcal{D}}(A(y^{k}-\bar{y})-t_{2}^{k}+u^{k}-\bar{u})\nonumber\\
		&=&A\bar{y}-b+\Pi_{\mathcal{D}}(A(y^{k}-\bar{y})-t_{2}^{k}+u^{k}-\bar{u})\label{Pi-C}
	\end{eqnarray}
	with $\mathcal{D}=\mathcal{T}_{C}(A\bar{y}-b)\cap \{\bar{u}\}^{\perp}$.
	Since $\mathcal{D}$ is a polyhedral cone, it follows from Proposition 4.1.4 of \cite{FacchineiPand2003} that there exist $l$ orthogonal projectors $E_{1},E_{2},\ldots,E_{l}$ such that
	\begin{eqnarray*}
		\Pi_{\mathcal{D}}(y)\in\{E_{1}y,E_{2}y,\ldots,E_{l}y\},\quad \quad\forall\, y\in\mathcal{R}^{n_{2}}.
	\end{eqnarray*}
	Therefore, we may further assume that there exists $1\leq i\leq l$ such that for all $k\geq 1$,
	\begin{eqnarray}
		\Pi_{\mathcal{D}}(A(y^{k}-\bar{y})-t_{2}^{k}+u^{k}-\bar{u})&=&E_{i}(A(y^{k}-\bar{y})-t_{2}^{k}+u^{k}-\bar{u})\nonumber\\
		&=&\Pi_{Range(E_{i})}(A(y^{k}-\bar{y})-t_{2}^{k}+u^{k}-\bar{u}).\label{Pi-Range}
	\end{eqnarray}
	Let $L=\mbox{Range}(E_{i})$ and we obtain from \eqref{Pi-C} and \eqref{Pi-Range} that
	\begin{eqnarray*}
		L\cap\mathcal{D}\ni (A(y^{k}-\bar{y})-t_{2}^{k})\ \perp\ (u^{k}-\bar{u})\in\mathcal L^{\perp}\cap {\mathcal{D}}^{\circ}.
	\end{eqnarray*}
	Since $A(y^{k}-\bar{y})-t_{2}^{k}=s_{k}Ad+o(s_{k})$ and $L\cap\mathcal{D}$ is a closed convex cone, we obtain $Ad\in L\cap\mathcal{D}$.
	
	In addition, by Assumption \ref{assumption} we know that for $t^{k}$ sufficiently small, $Dy^{k}-t_{3}^{k}\neq 0$ and we denote
	\begin{eqnarray*}
		v^{k}=\nabla q(Dy^{k}-t_{3}^{k})\quad\mbox{and}\quad \bar{v}=\nabla q(D\bar{y}).
	\end{eqnarray*}
	It follows that
	\begin{eqnarray*}
		\nabla q(Dy^{k}-t_{3}^{k})-\nabla q(D\bar{y})&=&\nabla q(D\bar{y}+s_{k}Dd)-\nabla q(D\bar{y})+\nabla q(Dy^{k}-t_{3}^{k})-\nabla q(D\bar{y}+s_{k}Dd)\\
		&=&s_{k}(\nabla q)'(D\bar{y};Dd)+o(s_{k})+O(\|Dy^{k}-D\bar{y}-s_{k}Dd-t_{3}^{k}\|)\\
		&=&s_{k}(\nabla q)'(D\bar{y};Dd)+o(s_{k}).
	\end{eqnarray*}
	Therefore, we have
	\begin{eqnarray*}
		0&=&s_{k}(\nabla q)'(D\bar{y};Dd)+o(s_{k})-(v^{k}-\bar{v})\\
		&=&s_{k}(\nabla q)'(D\bar{y};Dd)+o(s_{k})+D^{-T}A^{T}(u^{k}-\bar{u})-D^{-T}t_{1}^{k}\\
		&=&s_{k}(\nabla q)'(D\bar{y};Dd)+o(s_{k})+D^{-T}A^{T}(u^{k}-\bar{u}).
	\end{eqnarray*}
	Since $L^{\perp}\cap\mathcal{D}^{\circ}$ is a polyhedral cone, we know from Theorem 19.3 of \cite{Rockafellar1970} that $D^{-T}A^{T}(L^{\perp}\cap\mathcal{D}^{\circ})$ is still a polyhedral cone. There exists $\eta\in L^{\perp}\cap\mathcal{D}^{\circ}$ such that
	\begin{eqnarray}\label{primal-sosc}
		\langle Dd,(\nabla q)'(D\bar{y};Dd)\rangle=-\langle Dd,D^{-T}A^{T}\eta\rangle=-\langle Ad,\eta\rangle=0.
	\end{eqnarray}
	Note that
	\begin{eqnarray*}
		0=\langle Dd,(\nabla q)'(D\bar{y};Dd)\rangle=\frac{\|Dd\|^2}{\|D\bar{y}\|}-\frac{[(D\bar{y})^T(Dd)]^2}{\|D\bar{y}\|^3},
	\end{eqnarray*}
	thus it is obvious that (\ref{primal-sosc}) holds if and only if there exists $\lambda\in\mathcal{R}$ such that $d=\lambda\bar{y}$. Since
	\begin{eqnarray*}
		0=\langle Ad,\bar{u}\rangle=\langle d,A^{T}\bar{u}\rangle=-\langle d,D^{T}\bar{v}\rangle=-\langle Dd,\bar{v}\rangle=-\lambda\langle D\bar{y},\bar{v}\rangle=-\frac{\lambda}{\|D\bar{y}\|}\langle D\bar{y},D\bar{y}\rangle,
	\end{eqnarray*}
	we have $\lambda=0$, which is a contradiction with $\|d\|=1$.
	
	We have already proven that there exists a neighborhood
	$\mathcal{U}$ of $(\bar{y},\bar{u},\bar{v})$, the
	equality (\ref{eq:estimation}) is valid for any solution
	$(y(t),u(t),v(t))$ of the perturbed KKT system with
	$(y(t),u(t),v(t))\in\mathcal{U}$.
	
	Next we define the following two mappings
	\begin{eqnarray*}
		&&\Theta_{KKT}(y,u,v,t):=\left(\begin{array}{l}
			A^{T}u+D^{T}v-t_{1}\\
			u-\mbox{Prox}_{h^{*}}(Ay-b-t_{2}+u)\\
			v-\frac{Dy-t_{3}}{\|Dy-t_{3}\|}
		\end{array}\right),\\&&\quad\quad\quad\quad\quad\quad\quad\quad\forall\,
		(y,u,v,t)\in\mathcal{R}^{n_{1}}\times\mathcal{R}^{n_{2}}\times\mathcal{R}^{n_{1}}\times\mathcal{R}^{2n_{1}+n_{2}},\\
		&&\theta(u,v):=\Theta_{KKT}(\bar{y},u,v,0),\quad
		\forall\, (u,v)\in\mathcal{R}^{n_{2}}\times\mathcal{R}^{n_{1}}.
	\end{eqnarray*}
	Then we have $\theta(u,v)=0$ if and only if
	$(u,v)\in\mathcal{M}(\bar{y})$. Since
	$\mbox{Prox}_{h^{*}}(\cdot)$ is piecewise affine, $\theta(u,v)$ is a piecewise affine function and thus a polyhedral multifunction, therefore the error
    bound condition holds. Besides, $f(y)=\frac{y}{\|y\|}$ is Lipschitz continuous if $y\neq 0$. Hence, we can see
	that for any $t$ with its norm closely enough to zero and any
	solution $(y(t),u(t),v(t))\in\mathcal{U}$ associated with
	the perturbed KKT system, we have
	\begin{eqnarray*}
		&&\mbox{dist}((u(t),v(t)),\mathcal{M}(\bar{y}))\\
		&=&O(\|\theta(u(t),v(t))\|)\\
		&=&O(\|\Theta_{KKT}(\bar{y},u(t),v(t),0)-\Theta_{KKT}(y(t),u(t),v(t),t)\|)\\
		&=&O(\|y(t)-\bar{y}\|)+O(\|t\|).
	\end{eqnarray*}
	Together with the estimation (\ref{eq:estimation}), there exists a constant $\kappa>0$ such that
	\begin{eqnarray*}
		\|y(t)-\bar{y}\|+\mbox{dist}((u(t),v(t)),\mathcal{M}(\bar{y}))\leq\kappa\|t\|.
	\end{eqnarray*}
	Therefore the desired result follows.
\end{proof}

\subsection{An SsN method for solving the subproblem (\ref{eq:PALM})}
Note that it is essentially important to solve the subproblem (\ref{eq:PALM}) of PALM with a desired accuracy. In this subsection, we discuss how to apply the SsN method to obtain an approximate solution of the corresponding subproblem efficiently. For simplicity, we omit the superscript or subscript $k$. Given $\sigma>0$, $\tilde{y}\in\mathcal{R}^{n_{1}},\tilde{u}\in\mathcal{R}^{n_{2}},\tilde{v}\in\mathcal{R}^{n_{1}}$, we can rewrite the subproblem (\ref{eq:PALM}) as the following form
\begin{eqnarray}\label{subproblem-real}
	\min_{y\in\mathcal{R}^{n_{1}}}\phi(y),
\end{eqnarray}
where
\begin{eqnarray*}
	\phi(y)&:=& \varphi(y)+\frac{1}{2\sigma}\|y-\tilde{y}\|_{H}^{2},\\
	\varphi(y)&:=&\mathcal{L}_{\sigma}(y;\tilde{u},\tilde{v})=\sigma e_{\sigma^{-1}h}(Ay-b+\sigma^{-1}\tilde{u})+\sigma e_{\sigma^{-1}q}(Dy+\sigma^{-1}\tilde{v}).
\end{eqnarray*}
Since the function $\phi$ is strongly convex and smooth, finding the solution of problem (\ref{subproblem-real}) is equivalent to solving the following system of equations
\begin{eqnarray*}
	\nabla\phi(y)&=&\nabla\varphi(y)+\sigma^{-1}H(y-\tilde{y})=0,
\end{eqnarray*}
where
\begin{eqnarray*}
	\nabla\varphi(y)=A^{T}\mbox{Prox}_{\sigma h^{*}}(\sigma(Ay-b)+\tilde{u})+D^{T}\mbox{Prox}_{\sigma q^{*}}(\sigma Dy+\tilde{v}).
\end{eqnarray*}
Since the two mappings $\mbox{Prox}_{\sigma^{-1}h}(\cdot)$ and $\mbox{Prox}_{\sigma^{-1}q}(\cdot)$ are Lipschitz continuous if $q(\cdot)=\|\cdot\|_{1}$, $\|\cdot\|$ or $\|\cdot\|_{\infty}$, the following multifunction
\begin{eqnarray*}
	\hat{\partial}^{2}\varphi(y)&:=&\sigma A^{T}\partial\mbox{Prox}_{\sigma h^{*}}(\sigma(Ay-b)+\tilde{u})A+\sigma D^{T}\partial\mbox{Prox}_{\sigma q^{*}}(\sigma Dy+\tilde{v})D
\end{eqnarray*}
is well defined. It is known from \cite{Hiriart-Urruty} that
\begin{eqnarray*}
	\partial^{2}\varphi(y)d=\hat{\partial}^{2}\varphi(y)d,\ \forall\, d\in\mathcal{R}^{n_{1}}
\end{eqnarray*}
holds with $\partial^{2}\varphi(y)$ being the Clarke generalized Jacobian of $\varphi$ at $y$. We choose $U\in\partial\mbox{Prox}_{\sigma h^{*}}(\sigma(Ay-b)+\tilde{u})$ and $V\in\partial\mbox{Prox}_{\sigma q^{*}}(\sigma Dy+\tilde{v})$, then $\sigma A^{T}UA+\sigma D^{T}VD+\sigma^{-1}H\in\hat{\partial}^{2}\phi(y):=\hat{\partial}^{2}\varphi(y)+\sigma^{-1}H$.

Now we are ready to describe the SsN method and list the details as follows.
\bigskip

\noindent\fbox{\parbox{\textwidth}{\noindent{\bf Algorithm 2 (SsN): \label{alg:Newton-CG1}} Input $\sigma>0,$ $\tilde{y}\in\mathcal{R}^{n_{1}},\tilde{u}\in\mathcal{R}^{n_{2}},\tilde{v}\in\mathcal{R}^{n_{1}}$,
		$\mu\in(0,\frac{1}{2}), \ \bar{\eta}\in(0,1),
		\tau\in(0,1],\textrm{
			and }
		\delta\in(0,1)$. Choose $y^{0}\in\mathcal{R}^{n_{1}}$. Set $j=0$ and iterate:
		\begin{description}
			\item [Step 1.] Let $U^{j}\in\partial\mbox{Prox}_{\sigma h^{*}}(\sigma(Ay^{j}-b)+\tilde{u})$, $V^{j}\in\partial\mbox{Prox}_{\sigma q^{*}}(\sigma Dy^{j}+\tilde{v})$ and $H^{j}=\sigma A^{T}U^{j}A+\sigma D^{T}V^{j}D+\sigma^{-1}H$. Solve the following linear system
			\begin{eqnarray*}
				H^{j}\Delta y=-\nabla\phi(y^{j})
			\end{eqnarray*}
			by a direct method or the preconditioned conjugate gradient method to obtain an approximate solution $\Delta y^{j}$ satisfying the condition below
			\begin{eqnarray}
				\|H^{j}\Delta y^{j}+\nabla\phi(y^{j})\|\leq
				\eta_{j}:=\min(\bar{\eta},\|\nabla\phi(y^{j})\|^{1+\tau}).\label{ineq:stopping
					criterion 1}
			\end{eqnarray}
			\item [Step 2.]  Set $\alpha_{j}=\delta^{m_{j}}$, where $m_{j}$ is the first nonnegative integer $m$ such that
			\begin{eqnarray*}
				&\phi(y^{j}+\delta^{m}\Delta y^{j})\leq
				\phi(y^{j})+\mu\delta^{m}\langle\nabla
				\phi(y^{j}),\Delta y^{j}\rangle.&
			\end{eqnarray*}
			\item [Step 3.]  Set $y^{j+1}=y^{j}+\alpha_{j}\Delta y^{j}$. If a desired stopping criterion is satisfied, terminate; otherwise set $j=j+1$ and go to Step 1.
\end{description}}}

\bigskip

We need to mention that the proximal mappings $\mbox{Prox}_{\sigma^{-1}h}(\cdot)$ and $\mbox{Prox}_{\sigma^{-1}q}(\cdot)$ are strongly semismooth if $q(\cdot)=\|\cdot\|_{1}$ or $\|\cdot\|_{\infty}$ due to Proposition 7.4.7 of \cite{FacchineiPand2003}. So is the proximal mapping of $\|\cdot\|$ base on Theorem 4 of \cite{Meng2005}. It follows easily that the gradient function $\nabla\phi$ is strongly semismooth. This is crucial to guarantee the second order convergence rate of the SsN method. One may see e.g., \cite{LiSunToh2020} for more details. We just give the result below without proof.
\begin{theorem}
	Given $\sigma>0$, $\tilde{y}\in\mathcal{R}^{n_{1}}$, $\tilde{u}\in\mathcal{R}^{n_{2}}$ and $\tilde{v}\in\mathcal{R}^{n_{1}}$, the sequence $\{y^{j}\}$ generated by the SsN method converges to the unique solution $\bar{y}$ of $\nabla\phi(y)=0$ and it holds that
	\begin{eqnarray*}
		\|y^{j+1}-\bar{y}\|=\mathcal{O}(\|y^{j}-\bar{y}\|^{1+\tau}).
	\end{eqnarray*}
\end{theorem}

\subsection{The computational details of $\hat{\partial}^{2}\varphi(\cdot)$}
In this subsection, we discuss how to compute $\hat{\partial}^{2}\varphi(\cdot)$. In the implementation of the SsN method, for a given $\sigma>0$, we need to compute the generalized Jacobian matrices of the proximal mappings $\mbox{Prox}_{\sigma h^{*}}(u)$ and $\mbox{Prox}_{\sigma q^{*}}(v)$, respectively.

Each element $\Sigma$ of $\partial\mbox{Prox}_{\sigma h^{*}}(u)$ is a diagonal matrix with
\begin{eqnarray*}
	\Sigma_{ii}\in\left\{\begin{array}{ll}\{1\},&\mbox{if}\ u_{i}>0,\\
		\left[0,1\right],&\mbox{if}\ u_{i}=0,\\
		\{0\},&\mbox{otherwise}.\end{array}\right.
\end{eqnarray*}
As for the set $\partial\mbox{Prox}_{\sigma q^{*}}(v)$, we discuss it for different $q$.
\begin{enumerate}
	\item $q(\cdot)=\|\cdot\|_{1}$ or $q(\cdot)=\|\cdot\|_{\infty}$
	
	As is known from Section \ref{sec:preliminaries}, the function $\mbox{Prox}_{\sigma q^{*}}(v)$ is a metric projection function of a polyhedral convex set $C_{2}=\{v\ |\ \bar{A}v-\bar{b}\leq 0\}$. From \cite{Han1997,LiSunToh2020b}, we know that for any $v\in\mathcal{R}^{n_{1}}$
	\begin{eqnarray}\label{generalized-Jacobian}
		I_{n}-\bar{A}_{J}^{T}\left(\bar{A}_{J}\bar{A}_{J}^{T}\right)^{\dag}\bar{A}_{J}\in\partial\mbox{Prox}_{\sigma q^{*}}(v)
	\end{eqnarray}
	with $J:=\{j\ |\ \bar{A}_{j}v=\bar{b}_{j}\}$ and $(\cdot)^{\dag}$ being the Moore-Penrose pseudo-inverse.
	
	For more details, if $q(\cdot)=\|\cdot\|_{1}$ and $\Sigma\in\partial\mbox{Prox}_{\sigma q^{*}}(v)$, then $\Sigma$ is a diagonal matrix with
	\begin{eqnarray*}
		\Sigma_{ii}\in\left\{\begin{array}{ll}\{1\},&\mbox{if}\ |v_{i}|<1,\\
			\left[0,1\right],&\mbox{if}\ |v_{i}|=1,\\
			\{0\},&\mbox{otherwise}.\end{array}\right.
	\end{eqnarray*}
	If $q({\cdot})=\|\cdot\|_{\infty}$, then we can choose $V\in\partial\mbox{Prox}_{\sigma q^{*}}(v)$ such that
	\begin{eqnarray*}
		V=\left\{\begin{array}{ll}I_{N},&\mbox{if}\ \|v\|_{1}\leq 1,\\
			P_{v}\widetilde{V}P_{v},&\mbox{otherwise}\end{array}\right.
	\end{eqnarray*}
	with $\widetilde{V}=\textrm{diag}(s)-\frac{ss^{T}}{s^{T}s}\in\partial\Pi_{\triangle}(v)$ and $s_{i}=1$ if $(\Pi_{\triangle}(P_{v}v))_{i}\neq0$, otherwise $s_{i}=0$.
	
	\item $q(\cdot)=\|\cdot\|$
	
	We have
	\begin{eqnarray*}
		\partial\mbox{Prox}_{\sigma q^{*}}(v)&=&\left\{\begin{array}{ll}\{I_{n_{1}}\},&\quad\mbox{if}\quad\|v\|<1,\\
			\left\{I_{n_{1}}-\frac{tvv^{T}}{\|v\|^{2}},\ 0\leq t\leq 1\right\},&\quad\mbox{if}\quad\|v\|=1,\\
			\left\{\frac{1}{\|v\|}I_{n_{1}}-\frac{vv^{T}}{\|v\|^{3}}\right\},&\quad\mbox{if}\quad\|v\|>1.\end{array}\right.
	\end{eqnarray*}
\end{enumerate}

\subsection{DCGM}
In this subsection, we discuss how to deal with the large scale $O(n^{3})$ constraints. Since the number of the rows is so large that it is impossible to generate and store the entire matrix $A$. We apply DCGM to improve the efficiency of the implementation. The basic fact for the primal problem is that most of the constraints are not active at the solution point and therefore we can afford not to include those inactive constraint rows. Instead of dealing with all the constraints, the idea of DCGM is that we only consider a subset of the constraints and solve a reduced problem at each iteration.

We present DCGM below. One can see \cite{Bertsimas1997} for more details.

\bigskip

\noindent\fbox{\parbox{\textwidth}{\noindent{\bf Algorithm 3 (DCGM$\_$PALM): \label{dcgm-palm}} Choose $S$ as a subset of $\{1,2,\ldots,n_{2}\}$ with at least one element of $b_{S}$ is negative, where $b_{S}$ is the vector derived by $b$ with rows indexed by $S$. Iterate
		\begin{description}
			\item [Step 1.] Find an approximate solution $y^{*}$ of the following reduced problem
			\begin{eqnarray}\label{dcgm-palm-subproblem}
				\min_{y\in\mathcal{R}^{n_{1}}}\Big\{h(A_{S}y-b_{S})+q(Dy)\Big\}
			\end{eqnarray}
			by applying Algorithm 1, where $A_{S}$ denotes the matrix derived by $A$ with rows indexed by $S$.
			\item [Step 2.]  Check the feasibility of the remaining constraints and denote
			\begin{eqnarray*}
				S'=\Big\{j\, \Big|\, A_{j}y^{*}-b_{j}>0,\ j\in \{1,2,\ldots,n_{2}\}\setminus S\Big\}.
			\end{eqnarray*}
			\item [Step 3.]  If $S'={\O}$, terminate; otherwise update the set $S$ and go to Step 1.
\end{description}}}

\bigskip
\begin{remark}
	In the above algorithm, we add the first $\min\{|S|,|S'|\}$ biggest values of violated constraints into the constraint set $S$. As mentioned in \cite{Bertsimas1997}, we also take the idea of dropping some of the elements of $S$, i.e., we drop those constraints that are not active. In our implementation, we use the zero vector as the initial iteration to generate an initial subset $S$. We remove the constraints corresponding to the smallest $50\%$ values of $A_{j}y^{*}-b_{j}$ if $|S'|>|S|$ and $|S|>n(n-1)/4$. Due to the aforementioned strategy, the number of the elements in the constraint set $S$ increases during the iteration. Since the number of the total constraints is finite, DCGM can be terminated in a finite number of iterations.
\end{remark}

\section{Theoretical results}\label{sec:CQ-nondegeneracy-invertible}
Let $(\bar{x},\bar{u},\bar{v})$ be the solution of the KKT system (\ref{composite-problem-kkt}). In this section, we introduce some theoretical results on the SRCQ and the nondegeneracy condition of the dual problem at $(\bar{u},\bar{v})$, and the nonsingularity of $\hat{\partial}^{2}\varphi(\bar{y})$.

\subsection{The nondegeneracy condition and the nonsingularity of $\hat{\partial}^{2}\varphi(\bar{y})$}
In this subsection, we give a result about the equivalence between the dual nondegeneracy condition at $(\bar{u},\bar{v})$ and the nonsingularity of $\hat{\partial}^{2}\varphi(\bar{y})$. It is known that the nondegeneracy condition of the dual problem (\ref{dual-problem}) at $(\bar{u},\bar{v})$ takes the following form
\begin{eqnarray}\label{dual-nondegeneracy-cond}
	A^{T}\mbox{lin}\mathcal{T}_{C_{1}}(\bar{u})+D^{T}\mbox{lin}\mathcal{T}_{C_{2}}(\bar{v})=\mathcal{R}^{n_{1}}.
\end{eqnarray}
We present the theorem below.
\begin{theorem}
	Let $(\bar{y},\bar{u},\bar{v})$ be the solution of the KKT system. Then the nondegeneracy condition (\ref{dual-nondegeneracy-cond}) of the dual problem at $(\bar{u},\bar{v})$ holds if and only if $\hat{\partial}^{2}\varphi(\bar{y})$ is nonsingular.
\end{theorem}
\begin{proof}
	The nondegeneracy condition (\ref{dual-nondegeneracy-cond}) is equivalent to
	\begin{eqnarray*}
		\Big\{A^{T}\mbox{lin}\mathcal{T}_{C_{1}}(\bar{u})+D^{T}\mbox{lin}\mathcal{T}_{C_{2}}(\bar{v})\Big\}^{\perp}=\{0\},
	\end{eqnarray*}
	which is also equivalent to
	\begin{eqnarray*}
		\Big\{A^{T}\mbox{lin}\mathcal{T}_{C_{1}}(\bar{u})\Big\}^{\perp}\cap\Big\{D^{T}\mbox{lin}\mathcal{T}_{C_{2}}(\bar{v})\Big\}^{\perp}=\{0\}.
	\end{eqnarray*}
	Therefore, we have
	\begin{eqnarray*}
		\forall\, w^{1}\in\mbox{lin}\mathcal{T}_{C_{1}}(\bar{u}),\, w^{2}\in\mbox{lin}\mathcal{T}_{C_{2}}(\bar{v}),\,d\in\mathcal{R}^{n_{1}},\, \langle d,A^{T}w^{1}\rangle=0,\ \langle d,D^{T}w^{2}\rangle=0\\
		\Rightarrow\quad d=0,\qquad\qquad
	\end{eqnarray*}
	which can be written equivalently as
	\begin{eqnarray*}
		\forall\, w^{1}\in\mbox{lin}\mathcal{T}_{C_{1}}(\bar{u}),\ w^{2}\in\mbox{lin}\mathcal{T}_{C_{2}}(\bar{v}),\,\,d\in\mathcal{R}^{n_{1}},\, \langle Ad,w^{1}\rangle=0,\, \langle Dd,w^{2}\rangle=0\\
		\Rightarrow\quad d=0.\qquad\qquad
	\end{eqnarray*}
	
	According to the structure of $q$, we separate our proof into two parts.
	\begin{enumerate}
		\item $q(\cdot)=\|\cdot\|_{1}$ or $q(\cdot)=\|\cdot\|_{\infty}$.
		
		Let $J_{1} = \{j\, |\, \bar{u}_{j}>0\}$, $J_{2}=\{j\, |\, \bar{A}_{j}\bar{v}=\bar{b}_{j}\}$, $A_{J_{1}}$ and $\bar{A}_{J_{2}}$ are the matrices  derived from $A$ and $\bar{A}$, respectively. Since $A\bar{y}-b\in\partial h^{*}(\bar{u})$, we have $A_{J_{1}}\bar{y}-b_{J_{1}}=0$.
		Based on the structure of $\mathcal{T}_{C_{1}}(\bar{u})$ and $\mathcal{T}_{C_{2}}(\bar{v})$, the nondegeneracy condition (\ref{dual-nondegeneracy-cond}) is equivalent to
		\begin{eqnarray}\label{cond-nondegeneracy-equivalent}
			A_{J_{1}}d=0,\quad Dd\in\mbox{span}\Big\{\bar{A}_{j}^{T},\ j\in J_{2}\Big\}\quad \Rightarrow\quad d=0.
		\end{eqnarray}
		
		Furthermore, based on the Moreau identity we have
		\begin{eqnarray*}
			\mbox{Prox}_{\sigma^{-1}q}(D\bar{y}+\sigma^{-1}\bar{v})=D\bar{y}+\sigma^{-1}\bar{v}-\Pi_{\sigma^{-1}C_{2}}(D\bar{y}+\sigma^{-1}\bar{v}).
		\end{eqnarray*}
		By (\ref{generalized-Jacobian}), we have $I_{n_{1}}-\bar{A}_{J_{2}}^{T}[\bar{A}_{J_{2}}\bar{A}_{J_{2}}^{T}]^{\dag}\bar{A}_{J_{2}}\in\partial\mbox{Prox}_{\sigma q^{*}}(\sigma D\bar{y}+\bar{v})$.
		For any $\overline{U}\in\partial\mbox{Prox}_{\sigma h^{*}}(\sigma(A\bar{y}-b)+\bar{u})$, let $d\in\mathcal{R}^{n_{1}}$ such that
		\begin{eqnarray*}
			d^{T}A^{T}\overline{U}Ad+d^{T}D^{T}(I_{n_{1}}-\bar{A}_{J_{2}}^{T}[\bar{A}_{J_{2}}\bar{A}_{J_{2}}^{T}]^{\dag}\bar{A}_{J_{2}})Dd=0.
		\end{eqnarray*}
		By the formulation of $\partial\mbox{Prox}_{\sigma h^{*}}(\sigma(A\bar{y}-b)+\bar{u})$, we have
		\begin{eqnarray}\label{cond-AJ}
			A_{J_{1}}d=0.
		\end{eqnarray}
		For any vector $w\in\mathcal{R}^{n_{1}}$, it is known from the singular value decomposition of $\bar{A}_{J_{2}}$ and the Moore-Penrose pseudo-inverse of $\bar{A}_{J_{2}}\bar{A}_{J_{2}}^{T}$ that $\Pi_{\mbox{Ran}(\bar{A}_{J_{2}}^{T})}(w)=\bar{A}_{J_{2}}^{T}[\bar{A}_{J_{2}}\bar{A}_{J_{2}}^{T}]^{\dag}\bar{A}_{J_{2}}w$ and $(I_{n_{1}}-\bar{A}_{J_{2}}^{T}[\bar{A}_{J_{2}}\bar{A}_{J_{2}}^{T}]^{\dag}\bar{A}_{J_{2}})w\in\mbox{Ran}(\bar{A}_{J_{2}}^{T})^{\perp}=\mbox{Null}(\bar{A}_{J_{2}})$. There exist a unique $w^{1}\in\mbox{Ran}(\hat{A}_{J_{2}}^{T})$ and $w^{2}\in\mbox{Null}(\bar{A}_{J_{2}})$ such that $w = w^{1}+w^{2}$ with $\langle w^{1},w^{2}\rangle=0$. It follows that $\langle w,w^{2}\rangle=0$ is equivalent to $w\in\mbox{Ran}({\bar{A}_{J_{2}}^{T}})$. Then we have the conclusion that the nonsingularity of $\hat{\partial}^{2}\varphi(\bar{y})$ is equivalent to
		\begin{eqnarray*}
			A_{J_{1}}d=0,\quad Dd\in\mbox{Ran}({\bar{A}_{J_{2}}^{T}})\quad\Rightarrow\quad d=0,
		\end{eqnarray*}
		which is the same as that in (\ref{cond-nondegeneracy-equivalent}). The desired result follows.
		\item $q(\cdot)=\|\cdot\|$.
		
		Firstly, we know that $\|\bar{v}\|\leq 1$ due to $\bar{v}\in\partial q(D\bar{y})$.
		If $\|\bar{v}\|<1$, then the nondegeneracy condition (\ref{dual-nondegeneracy-cond}) is equivalent to
		\begin{eqnarray*}
			A_{J_{1}}d=0,\quad Dd=0\quad\Rightarrow\quad d=0,
		\end{eqnarray*}
		which is natural since $D$ is nonsingular. Then the Hessian of $\varphi$ at $\bar{y}$ is $\sigma A_{J_{1}}^{T}A_{J_{1}}+\sigma D^{T}D$, which is nonsingular.
		
		If $\|\bar{v}\|=1$, the nondegeneracy condition (\ref{dual-nondegeneracy-cond}) holds if and only if
		\begin{eqnarray}\label{dual-nondegeneracy-cond-equivalent-l2}
			A_{J_{1}}d=0,\quad Dd//\bar{v}\quad\Rightarrow\quad d=0.
		\end{eqnarray}
		Note that $D\bar{y}=\lambda\bar{v}$, $\lambda>0$ due to the third equation of the KKT condition \eqref{composite-problem-kkt} and Assumption \ref{assumption}.  Hence, for any $t\in[0,1]$,
		\begin{eqnarray*}
			\overline{V}&:=&\frac{1}{\|\sigma D\bar{y}+\bar{v}\|}I_{n_{1}}-\frac{(t\sigma D\bar{y}+\bar{v})(D\bar{y}+\sigma^{-1}\bar{v})^{T}}{\|\sigma D\bar{y}+\bar{v}\|^{3}}\\
			&=&\frac{1}{\|\sigma D\bar{y}+\bar{v}\|}I_{n_{1}}-\frac{t\bar{v}\bar{v}^{T}}{\|\sigma D\bar{y}+\bar{v}\|}\in\partial\mbox{Prox}_{\sigma q^{*}}(\sigma D\bar{y}+\bar{v}).
		\end{eqnarray*}
		For any $\overline{U}\in\partial\mbox{Prox}_{\sigma h^{*}}(\sigma(A\bar{y}-b)+\bar{u})$, we can write an element in $\hat{\partial}^{2}\varphi(\bar{y})$ as
		\begin{eqnarray}\label{hessian-l2}
			\sigma A^{T}\overline{U}A+\sigma D^{T}\overline{V}D=\sigma A^{T}\overline{U}A+\frac{1}{\|D\bar{y}+\sigma^{-1}\bar{v}\|} D^{T}\left(I_{n_{1}}-t\bar{v}\bar{v}^{T}\right)D.
		\end{eqnarray}
		Combining (\ref{cond-AJ}), the matrix (\ref{hessian-l2}) is nonsingular if and only if
		\begin{eqnarray*}
			\forall\ t\in[0,1],\ A_{J_{1}}d=0,\ \mbox\ d^{T}D^{T}(I_{n_{1}}-t\bar{v}\bar{v}^{T})Dd=0\quad\Rightarrow\quad d=0,
		\end{eqnarray*}
		which implies the result (\ref{dual-nondegeneracy-cond-equivalent-l2}).
	\end{enumerate}
	This completes the proof.
\end{proof}

\begin{remark}\label{remk4}
	One point we need to mention is that the nondegeneracy condition (\ref{dual-nondegeneracy-cond}) holds naturally in the case $q(\cdot)=\|\cdot\|$ under Assumption \ref{assumption}. It is obvious that $\|\bar{v}\|=1$ and $\bar{v}=\frac{D\bar{y}}{\|D\bar{y}\|}$ due to the condition $D\bar{y}\in\partial q^{*}(\bar{v})$. From the above proof, we know that the nodegeneracy condition (\ref{dual-nondegeneracy-cond}) is equivalent to
	\begin{eqnarray*}
		A_{J_{1}}d=0,\quad Dd // \bar{v}\quad\Rightarrow\quad d=0.
	\end{eqnarray*}
	Therefore, if the left-hand side of the above condition holds, there exists $\gamma\in\mathcal{R}$ such that $d=\gamma\bar{y}$ and $0=\langle Ad,\bar{u}\rangle=\gamma\langle A\bar{y},\bar{u}\rangle$. Since
	\begin{eqnarray*}
		\langle A\bar{y},\bar{u}\rangle=\langle\bar{y},A^{T}\bar{u}\rangle=-\langle\bar{y},D^{T}\bar{v}\rangle=-\left\langle D\bar{y},\frac{D\bar{y}}{\|D\bar{y}\|}\right\rangle=-\|D\bar{y}\|\neq0,
	\end{eqnarray*}
	we obtain $\gamma=0$ and the Hessian of $\varphi$ at $\bar{y}$ is nonsingular under Assumption \ref{assumption} if $q(\cdot)=\|\cdot\|$.
\end{remark}

\subsection{The equivalence between the SRCQ and the nondegeneracy condition when $q(\cdot)=\|\cdot\|_{1}$ or $q(\cdot)=\|\cdot\|_{\infty}$}
As is known from the previous subsection, we should give a good description about the dual nondegeneracy condition when $q(\cdot)=\|\cdot\|_{1}$ or $q(\cdot)=\|\cdot\|_{\infty}$. In this subsection, we consider the equivalence between the SRCQ of the dual problem (\ref{dual-problem}), which has the following form
\begin{align}\label{dual-strict-robinson-qualification}
	A^{T}\Big\{d\,\Big|\,h^{*\downarrow}_{-}(\bar{u},d)=\langle A\bar{y}-b,d\rangle\Big\}+D^{T}\Big\{d\,\Big|\,q^{*\downarrow}_{-}(\bar{v},d)=\langle D\bar{y},d\rangle\Big\}=\mathcal{R}^{n_{1}},
\end{align}
and the nondegeneracy condition (\ref{dual-nondegeneracy-cond}) in the case of $q(\cdot)=\|\cdot\|_{1}$ or $\|\cdot\|_{\infty}$. We first present some basic results for further use.
\begin{proposition}\label{prop-equivalent-h-p}
	For $h(\cdot) = \delta_{C}(\cdot)$, $q(\cdot)=\|\cdot\|_{1}$ or $q(\cdot)=\|\cdot\|_{\infty}$,
	we have
	\begin{eqnarray}
		&&\Big\{d\, \Big|\, h^{*\downarrow}_{-}(\bar{u},d)=\langle A\bar{y}-b,d\rangle\Big\}^{\circ}=\Big\{d\, \Big|\, h^{\downarrow}_{-}(A\bar{y}-b,d)=\langle\bar{u},d\rangle\Big\},\label{polar-p*-equivalent-a}\\
		&&\Big\{d\, \Big|\, q^{*\downarrow}_{-}(\bar{v},d)=\langle D\bar{y},d\rangle\Big\}^{\circ}=\Big\{d\, \Big|\, q^{\downarrow}_{-}(D\bar{y},d)=\langle\bar{v},d\rangle\Big\}.\label{polar-p*-equivalent-b}
	\end{eqnarray}
\end{proposition}
\begin{proof}
	We only need to prove (\ref{polar-p*-equivalent-b}), since the validity of (\ref{polar-p*-equivalent-a}) follows easily.
	For $q(\cdot)=\|\cdot\|_{1}$ or $q(\cdot)=\|\cdot\|_{\infty}$, we have
	\begin{eqnarray*}
		\Big\{d\, \Big|\, q^{*\downarrow}_{-}(\bar{v},d)=\langle D\bar{y},d\rangle\Big\}=\mathcal{T}_{C_{2}}(\bar{v})\cap\{D\bar{y}\}^{\perp}
	\end{eqnarray*}
	and
	\begin{eqnarray*}
		\Big\{d\, \Big|\, q^{*\downarrow}_{-}(\bar{v},d)=\langle D\bar{y},d\rangle\Big\}^{\circ}=\mathcal{T}_{C_{2}}(\bar{v})^{\circ}+\{D\bar{y}\}=\mathcal{N}_{C_{2}}(\bar{v})+\{D\bar{y}\},
	\end{eqnarray*}
	which shows that for any $d\in\Big\{d\, \Big|\, q^{*\downarrow}_{-}(\bar{v},d)=\langle D\bar{y},d\rangle\Big\}^{\circ}$, there exist $\alpha\in\mathcal{R}$ and $\lambda_{i}\geq0$ for $i\in\mathcal{I}(\bar{v})$ such that $d=\sum\limits_{i\in\mathcal{I}(\bar{v})}\lambda_{i}\bar{A}^{T}_{i}+\alpha D\bar{y}$.
	
	Since $D\bar{y}\in\mathcal{N}_{C_{2}}(\bar{v})$, there exists $\beta_{i}\geq0$ for $i\in\mathcal{I}(\bar{v})$ such that $D\bar{y}=\sum\limits_{i\in\mathcal{I}(\bar{v})}\beta_{i}\bar{A}^{T}_{i}$. From
	\begin{eqnarray*}
		d\in\partial\delta^{*}_{C_{2}}(D\bar{y})&\Leftrightarrow& d\in\mathop{\mbox{argmax}}_{z}\Big\{\langle D\bar{y},z\rangle-\delta_{C_{2}}(z)\Big\}\\
		&\Leftrightarrow& d\in\mathop{\mbox{argmax}}_{z}\Big\{\langle D\bar{y},z\rangle\, \Big|\, \bar{A}z-\bar{b}\leq0\Big\},
	\end{eqnarray*}
	we obtain
	\begin{align}\label{eq:partial-delta*}
		\partial\delta^{*}_{C_{2}}(D\bar{y})&=\mathop{\mbox{argmax}}_{z}\left\{\langle D\bar{y},z\rangle\, \Big|\, \bar{A}z-\bar{b}\leq0\right\}\nonumber\\
		&=\left\{z\ \left|\begin{array}{ll}\langle \bar{A}_{i}^{T},z\rangle-\bar{b}_{i}=0,& \mbox{if } i\in\mathcal{I}(\bar{v})\cap\{j\ |\ \beta_{j}>0\},\\
			\langle \bar{A}_{i}^{T},z\rangle-\bar{b}_{i}\leq0,& \mbox{if } i\notin\mathcal{I}(\bar{v}),\mbox{or}\ i\in\mathcal{I}(\bar{v})\cap\{j\ |\  \beta_{j}=0\}.
		\end{array}\right.\right\}.
	\end{align}
	Therefore, it follows that
	\begin{align}\label{eq:delta-sup}
		(\delta^{*}_{C_{2}})^{\downarrow}_{-}(D\bar{y},d)&=\sup_{z}\Big\{\langle d,z\rangle\ \Big|\ z\in\partial\delta^{*}_{C_{2}}(D\bar{y})\Big\}\nonumber\\
		&=\sup_{z}\left\{\langle d,z\rangle\ \left|\begin{array}{ll}\langle \bar{A}_{i}^{T},z\rangle-\bar{b}_{i}=0,& i\in\mathcal{I}(\bar{v})\cap\{j\ |\ \beta_{j}>0\}\\
			\langle \bar{A}_{j}^{T},z\rangle-\bar{b}_{i}\leq0,& i\notin\mathcal{I}(\bar{v}),\mbox{or}\ i\in\mathcal{I}(\bar{v})\cap\{j\ |\ \beta_{j}=0\}.
		\end{array}\right.\right\}.
	\end{align}
	By the KKT condition of the optimization problem (\ref{eq:delta-sup}), there exists a $\widetilde{\beta}$ with $\widetilde{\beta}_{i}\in\mathcal{R}$ if $i\in\mathcal{I}(\bar{v})\cap\{j\ |\ \beta_{j}>0\}$, $\widetilde{\beta}_{i}\geq0$ if $i\notin\mathcal{I}(\bar{v})$ or $i\in\mathcal{I}(\bar{v})\cap\{j\ |\ \beta_{j}=0\}$ such that $d=\bar{A}^{T}\widetilde{\beta}$. Then we can obtain
	\begin{eqnarray*}
		q^{\downarrow}_{-}(D\bar{y},d)=(\delta^{*}_{C_{2}})^{\downarrow}_{-}(D\bar{y},d)=\left\{\begin{array}{ll}\langle \widetilde{\beta},\bar{b}\rangle,&\mbox{if}\ d=\bar{A}^{T}\widetilde{\beta},\\+\infty,& \mbox{otherwise.}\end{array}\right.
	\end{eqnarray*}
	Therefore,
	\begin{eqnarray*}
		\Big\{d\, \Big|\, q^{\downarrow}_{-}(D\bar{y},d)=\langle\bar{v},d\rangle\Big\}
		&=&\left\{d\, \left|\, \langle\widetilde{\beta},\bar{b}\rangle=\langle\bar{v},d\rangle\right.\right\}\\
		&=&\left\{d\, \left|\, \langle\widetilde{\beta},\bar{b}\rangle=\langle\bar{v},\bar{A}^{T}\widetilde{\beta}\rangle,\ d=\bar{A}^{T}\widetilde{\beta}\right.\right\}\\
		&=&\left\{d\, \left|\, \sum_{i=1}^{n}\widetilde{\beta}_{i}\bar{b}_{i}=\sum_{i=1}^{n}\widetilde{\beta}_{i}\langle \bar{A}_{i}^{T},\bar{v}\rangle,\ d=\bar{A}^{T}\widetilde{\beta}\right.\right\}\\
		&=&\left\{d\, \left|\, \sum_{i=1}^{n}\widetilde{\beta}_{i}\bar{b}_{i}=\sum_{i\in\mathcal{I}(\bar{v})}\widetilde{\beta}_{i}\bar{b}_{i}+\sum_{i\notin\mathcal{I}(\bar{v})}\widetilde{\beta}_{i}\langle \bar{A}_{i}^{T},\bar{v}\rangle,\ d=\bar{A}^{T}\widetilde{\beta}\right.\right\}\\
		&=&\left\{d\, \left|\, d=\sum_{i\in\mathcal{I}(\bar{v})}\widetilde{\beta}_{i}\bar{A}_{i}^{T}\right.\right\}\\
		&=&\left\{d\, \left|\, d=\widetilde{\gamma} D\bar{y}+\sum_{i\in\mathcal{I}(\bar{v})}\widetilde{\widetilde{\beta}}_{i}\bar{A}_{i}^{T}\right.\right\}\\
		&=&\mbox{span}(D\bar{y})+\mathcal{N}_{C_{2}}(\bar{v}),
	\end{eqnarray*}
	where $\widetilde{\widetilde{\beta}}_{i}\geq0$ for $i\in\mathcal{I}(\bar{v})$, $\widetilde{\gamma}=\min\limits_{{i\in\mathcal{I}(\bar{v})}\atop\{i\,|\,\beta_{i}>0\}}\left\{\frac{\widetilde{\beta_{i}}}{\beta_{i}}\right\}$, and the fifth equality of the above formula is because $\sum\limits_{i\notin\mathcal{I}(\bar{v})}\widetilde{\beta}_{i}(\langle \bar{A}_{i}^{T},\bar{v}\rangle-\bar{b}_{i})=0$ implies $\widetilde{\beta}_{i}=0\, (i\notin\mathcal{I}(\bar{v}))$. Then we obtain
	\begin{align}\label{eq:theorem5-eq-2}
		\Big\{d\ \Big|\ q^{*\downarrow}_{-}(D\bar{y},d)=\langle\bar{v},d\rangle\Big\}^{\circ}=\mathcal{N}_{C_{2}}(\bar{v})+\{D\bar{y}\}=\Big\{d\ |\ q^{\downarrow}_{-}(D\bar{y},d)=\langle\bar{v},d\rangle\Big\}.
	\end{align}
	This completes the proof.
\end{proof}

\begin{theorem}\label{SRCQ-critical-0}
	For $h(\cdot) = \delta_{C}(\cdot)$, $q(\cdot)=\|\cdot\|_{1}$ or $q(\cdot)=\|\cdot\|_{\infty}$,
	the SRCQ (\ref{dual-strict-robinson-qualification}) of the dual problem
	holds if and only if the critical cone of the primal problem \eqref{primal-problem}
	\begin{eqnarray*}
		\mathcal{C}(\bar{y})=\Big\{d\ \Big|\ h^{\downarrow}_{-}(A\bar{y}-b,Ad)+q^{\downarrow}_{-}(D\bar{y},Dd)=0\Big\}
	\end{eqnarray*}
	contains only the zero element.
\end{theorem}
\begin{proof}
	Computing the polar cone of both sides of (\ref{dual-strict-robinson-qualification}), we obtain that
	\begin{eqnarray*}
		\left(A^{T}\Big\{d\ \Big|\ h^{*\downarrow}_{-}(\bar{u},d)=\langle A\bar{y}-b,d\rangle\Big\}+D^{T}\Big\{d\ \Big|\ q^{*\downarrow}_{-}(\bar{v},d)=\langle D\bar{y},d\rangle\Big\}\right)^{\circ}=\{0\}.
	\end{eqnarray*}
	Since
	\begin{eqnarray*}
		&&\left(A^{T}\Big\{d\, \Big|\, h^{*\downarrow}_{-}(\bar{u},d)=\langle A\bar{y}-b,d\rangle\Big\}+D^{T}\Big\{d\, \Big|\, q^{*\downarrow}_{-}(\bar{v},d)=\langle D\bar{y},d\rangle\Big\}\right)^{\circ}\\
		&=&\left(A^{T}\Big\{d\, \Big|\, h^{*\downarrow}_{-}(\bar{u},d)=\langle A\bar{y}-b,d\rangle\Big\}\right)^{\circ}\cap\left(D^{T}\Big\{d\, \Big|\, q^{*\downarrow}_{-}(\bar{v},d)=\langle D\bar{y},d\rangle\Big\}\right)^{\circ}
	\end{eqnarray*}
	and
	\begin{eqnarray*}
		&&\forall\, d\in\left(A^{T}\Big\{d\, \Big|\, h^{*\downarrow}_{-}(\bar{u},d)=\langle A\bar{y}-b,d\rangle\Big\}\right)^{\circ}\ \Leftrightarrow\  Ad\in\Big\{d\, \Big|\, h^{*\downarrow}_{-}(\bar{u},d)=\langle A\bar{y}-b,d\rangle\Big\}^{\circ},\\
		&&\forall\, d\in\left(D^{T}\Big\{d\, \Big|\, q^{*\downarrow}_{-}(\bar{v},d)=\langle D\bar{y},d\rangle\Big\}\right)^{\circ}\quad\Leftrightarrow\quad Dd\in\Big\{d\, \Big|\, q^{*\downarrow}_{-}(\bar{v},d)=\langle D\bar{y},d\rangle\Big\}^{\circ},
	\end{eqnarray*}
	it follows from Proposition \ref{prop-equivalent-h-p} that (\ref{dual-strict-robinson-qualification}) is equivalent to
	\begin{eqnarray*}
		\Big\{d\, \Big|\, h^{\downarrow}_{-}(A\bar{y}-b,Ad)=\langle\bar{u},Ad\rangle\Big\}\cap\Big\{d\, \Big|\, q^{\downarrow}_{-}(D\bar{y},Dd)=\langle\bar{v},Dd\rangle\Big\}=\{0\}.
	\end{eqnarray*}
	By the KKT condition \eqref{composite-problem-kkt} and the property of the directional derivative (since $h$ and $q$ are Lipschitz continuous), we have
	\begin{eqnarray}
		&&\Big\{d\, \Big|\, h^{\downarrow}_{-}(A\bar{y}-b,Ad)+q^{\downarrow}_{-}(D\bar{y},Dd)=0\Big\}\nonumber\\
		&=&\Big\{d\, \Big|\, h^{\downarrow}_{-}(A\bar{y}-b,Ad)=\langle\bar{u},Ad\rangle,\ q^{\downarrow}_{-}(D\bar{y},Dd)=\langle\bar{v},Dd\rangle\Big\}.\label{eq:theorem6-eq-1}
	\end{eqnarray}
	The desired result follows.
\end{proof}

\begin{proposition}\label{proposition-ri} Suppose Assumption \ref{assumption} holds. For $h(\cdot) = \delta_{C}(\cdot)$, $q(\cdot)=\|\cdot\|_{1}$ or $q(\cdot)=\|\cdot\|_{\infty}$,
	let $(\bar{y},\bar{u},\bar{v})$ be a solution of the KKT system.
	Then
	\begin{eqnarray*}
		&&\bar{u}\in\mbox{ri}(\partial h(A\bar{y}-b))\quad\Leftrightarrow\quad A\bar{y}-b\in\mbox{ri}(\mathcal{N}_{C_{1}}(\bar{u})),\\
		&&\bar{v}\in\mbox{ri}(\partial q(D\bar{y}))\quad\Leftrightarrow\quad D\bar{y}\in\mbox{ri}(\mathcal{N}_{C_{2}}(\bar{v})).
	\end{eqnarray*}
\end{proposition}
\begin{proof}
	We only give a proof for the second equivalence and the first one can be proved similarly. Since $D\bar{y}\in\mathcal{N}_{C_{2}}(\bar{v})$, there exists $\beta_{i}\geq0$ for $i\in\mathcal{I}(\bar{v})$ such that $D\bar{y}=\sum\limits_{i\in\mathcal{I}(\bar{v})}\beta_{i}\bar{A}^{T}_{i}$ by Proposition \ref{prop-Tangent-Normal-Cone}. Assumption \ref{assumption} implies that $\mathcal{I}(\bar{v})\neq\emptyset$. Due to (\ref{eq:partial-delta*}) we have that
	\begin{eqnarray*}
		&&\mbox{ri}(\partial q(D\bar{y}))\\
		&&=\mbox{ri}\left(\left\{z\, \left|\,\begin{array}{ll}\langle \bar{A}_{i}^{T},z\rangle-\bar{b}_{i}=0,& \mbox{if } i\in\mathcal{I}(\bar{v})\cap\{j\, |\, \beta_{j}>0\},\\
			\langle \bar{A}_{i}^{T},z\rangle-\bar{b}_{i}\leq0,& \mbox{if } i\notin\mathcal{I}(\bar{v}),\mbox{or}\ i\in\mathcal{I}(\bar{v})\cap\{j\ |\  \beta_{j}=0\}.
		\end{array}\right.\right\}\right)\\&&=\left\{z\, \left|\,\begin{array}{ll}\langle \bar{A}_{i}^{T},z\rangle-\bar{b}_{i}=0,& \mbox{if } i\in\mathcal{I}(\bar{v})\cap\{j\ |\ \beta_{j}>0\},\\
			\langle \bar{A}_{i}^{T},z\rangle-\bar{b}_{i}<0,& \mbox{if } i\notin\mathcal{I}(\bar{v}),\mbox{or}\ i\in\mathcal{I}(\bar{v})\cap\{j\ |\  \beta_{j}=0\}.
		\end{array}\right.\right\},
	\end{eqnarray*}
	which implies that $\bar{v}\in\mbox{ri}(\partial q(D\bar{y}))$ is equivalent to $\mathcal{I}(\bar{v})\cap\{j\ |\  \beta_{j}=0\}=\emptyset$.
	Therefore, the desired result follows from the fact that $D\bar{y}\in\mbox{ri}(\mathcal{N}_{C_{2}}(\bar{v}))$ is equivalent to $D\bar{y}=\sum\limits_{i\in\mathcal{I}(\bar{v})}\beta_{i}\bar{A}^{T}_{i}$ with $\beta_{i}>0$ for all $i\in\mathcal{I}(\bar{v})$.
\end{proof}

\begin{proposition} Suppose that Assumption \ref{assumption} is valid. For $h(\cdot) = \delta_{C}(\cdot)$, $q(\cdot)=\|\cdot\|_{1}$ or $q(\cdot)=\|\cdot\|_{\infty}$,
	the SRCQ (\ref{dual-strict-robinson-qualification}) is the same as the nondegeneracy condition (\ref{dual-nondegeneracy-cond}) for the dual problem \eqref{dual-problem}.
\end{proposition}
\begin{proof}
	If the SRCQ (\ref{dual-strict-robinson-qualification}) holds, we have $A\bar{y}-b\in\mbox{ri}(\mathcal{N}_{C_{1}}(\bar{u}))$ and $D\bar{y}\in\mbox{ri}(\mathcal{N}_{C_{2}}(\bar{v}))$ due to Theorem \ref{SRCQ-critical-0}, Proposition \ref{prop-critical-cone-linearspace-dualvalue} and Proposition \ref{proposition-ri}. Based on Proposition 4.73 of \cite{Bonnans2000}, we obtain
	\begin{eqnarray*}
		&&\mbox{lin}(\mathcal{T}_{C_{1}}(\bar{u}))=\mathcal{T}_{C_{1}}(\bar{u})\ \cap\ \{A\bar{y}-b\}^{\perp},\\
		&&\mbox{lin}(\mathcal{T}_{C_{2}}(\bar{v}))=\mathcal{T}_{C_{2}}(\bar{v})\ \cap\ \{D\bar{y}\}^{\perp},
	\end{eqnarray*}
	and complete the proof.
\end{proof}

\begin{proposition}For $h(\cdot) = \delta_{C}(\cdot)$, $q(\cdot)=\|\cdot\|_{1}$ or $q(\cdot)=\|\cdot\|_{\infty}$,
	the critical cone $\mathcal{C}(\bar{y})=\{0\}$ if and only if the solution $\bar{y}$ of the primal problem \eqref{primal-problem} is unique.
\end{proposition}
\begin{proof}
	We first prove that the critical cone $\mathcal{C}(\bar{y})=\{0\}$ implies the solution of the primal problem is unique by contradiction. Let $\bar{y}$ and $y'$ be two different solutions of the primal problem \eqref{primal-problem}, and $u',v'$ are the dual solutions corresponding to $y'$. Then by the KKT condition \eqref{composite-problem-kkt}, we have
	\begin{eqnarray*}
		&&A^{T}\bar{u}+D^{T}\bar{v}=0,\ \bar{u}\in\partial h(A\bar{y}-b),\ \bar{v}\in\partial q(D\bar{y}),\\
		&&A^{T}u'+D^{T}v' =0,\ u'\in\partial h(Ay'-b),\ v'\in\partial q(Dy').
	\end{eqnarray*}
	Let $dy=y'-\bar{y}$, $du=u'-\bar{u}$ and $dv=v'-\bar{v}$, we have
	\begin{eqnarray}\label{critical-cone-unique-solution-1}
		A^{T}du+D^{T}dv=0.
	\end{eqnarray}
	Since
	\begin{eqnarray*}
		&&\bar{u}=\Pi_{C_{1}}(A\bar{y}-b+\bar{u}),\ \bar{v}=\Pi_{C_{2}}(D\bar{y}+\bar{v}),\\
		&&u'=\Pi_{C_{1}}(Ay'-b+u'),\ v'=\Pi_{C_{2}}(Dy'+v'),
	\end{eqnarray*}
	by Theorem 4.1.1 of \cite{FacchineiPand2003}, it follows that
	\begin{eqnarray}\label{critical-cone-unique-solution-2}
		du=\Pi_{\mathcal{T}_{C_{1}}(\bar{u})\cap\{A\bar{y}-b\}^{\perp}}(Ady+du),\ dv = \Pi_{\mathcal{T}_{C_{2}}(\bar{v})\cap\{D\bar{y}\}^{\perp}}(Ddy+dv),
	\end{eqnarray}
	which is equivalent to
	\begin{align}\label{critical-cone-unique-solution-3}
		Ady=\Pi_{\mathcal{N}_{C_{1}}(\bar{u})+\{A\bar{y}-b\}}(Ady+du),\ Ddy = \Pi_{\mathcal{N}_{C_{2}}(\bar{v})+\{D\bar{y}\}}(Ddy+dv).
	\end{align}
	Combing (\ref{critical-cone-unique-solution-2}), (\ref{critical-cone-unique-solution-3}), \eqref{eq:theorem5-eq-2} and \eqref{eq:theorem6-eq-1}, we obtain that $dy\in\mathcal{C}(\bar{y})$ and therefore $dy=0$, which is a contradiction.
	
	Conversely, assume the solution of the primal problem \eqref{primal-problem} is unique. Let $\bar{y}$ be this unique solution. Suppose that there exists a vector $dy\in\mathcal{C}(\bar{y})$ but $dy\neq0$. Therefore, we obtain
	\begin{eqnarray*}
		&&Ady\in\Big\{d\ \Big|\ h^{\downarrow}_{-}(A\bar{y}-b,d)=\langle \bar{u},d\rangle \Big\}=\mathcal{N}_{C_{1}}(\bar{u})+\{A\bar{y}-b\},\\
		&&Ddy\in\Big\{d\ \Big|\ q^{\downarrow}_{-}(D\bar{y},d)=\langle \bar{v},d\rangle \Big\}=\mathcal{N}_{C_{2}}(\bar{v})+\{D\bar{y}\}.
	\end{eqnarray*}
	Furthermore, due to the fact that $\Pi_{E}(z)=0$ if $E$ is a cone and $z\in E^{\circ}$, we have
	\begin{eqnarray*}
		\Pi_{C_{1}}(A(\bar{y}+dy)-b+\bar{u})&=&\Pi_{C_{1}}(A\bar{y}-b+\bar{u})+\Pi_{\mathcal{T}_{C_{1}}(\bar{u})\cap\{A\bar{y}-b\}^{\perp}}(Ady)=\bar{u},\\
		\Pi_{C_{2}}(D(\bar{y}+dy)+\bar{v})&=&\Pi_{C_{2}}(D\bar{y}+\bar{v})+\Pi_{\mathcal{T}_{C_{2}}(\bar{v})\cap\{D\bar{y}\}^{\perp}}(Ddy)=\bar{v},
	\end{eqnarray*}
	which implies that
	\begin{eqnarray*}
		A(\bar{y}+dy)-b\in\partial h^{*}(\bar{u}),\quad D(\bar{y}+dy)\in\partial q^{*}(\bar{v}),
	\end{eqnarray*}
	and $(\bar{y}+dy,\bar{u},\bar{v})$ is also a solution of the KKT system. This is a contradiction of the uniqueness of the primal solution. The desired result follows.
\end{proof}
Now we conclude the following result.
\begin{theorem}
	Suppose that Assumption \ref{assumption} is valid. For $h(\cdot) = \delta_{C}(\cdot)$, $q(\cdot)=\|\cdot\|_{1}$ or $q(\cdot)=\|\cdot\|_{\infty}$,
	the nondegeneracy condition (\ref{dual-nondegeneracy-cond}) of the dual problem holds if and only if that the solution of the primal problem \eqref{primal-problem} is unique.
\end{theorem}

\section{The computational issues of DCGM$\_$PALM}\label{sec:computing-issues}
In this section, we discuss some important computational issues related to our algorithm.
\begin{enumerate}
	\item One of the difficulties we need to conquer is how to deal with the storage problem of the constraint matrix and the dual variable. Due to the special structure of the constraint matrix $A$, we only need to know the column index triple of the three nonzeros for each row. For a given row index, we can find the column index triple of the corresponding nonzeros and the column index whose coefficient takes $-1$, while others take $1$.
	We do not need to store the constraint matrix $A$ at all. As the number of active constraints is much less than $n_{2}$ and most elements of the dual variable are zeros due to the complementary condition, the memory requirement is acceptable for sparse restoration. As for large sparse graphs, we also store the vectors $b$ and $\mbox{trivec}(W)$ in a sparse way.
	
	\item Before the iteration of DCGM$\_$PALM, it is important to choose an initial violated constraints set which is as small as possible but covers the true active constraints with high probability. Let $X$ be a feasible solution of the original problem (\ref{metric-nearness-problem}). Given indices $1\leq i<j<k\leq n$, it is easy to know that $x_{ij}\geq 0$ and there is only one violated constraint related to the triple $(x_{ij},x_{ik},x_{jk})$ if it exits. Fortunately, $y=0$ is a natural guess solution and we can find an initial violated constraint set including the constraints whose corresponding components of $b$ are less than 0.
	
	\item At the $k$-th iteration, we apply PALM to solve the reduced optimization problem with $n_{1}$ variables and $|S^{k}|$ constraints, which is still a huge problem when $n$ is large. Let $I^{k}$ be the set of the indices of variables involved in the constraints and $y^{k}$ be the approximate optimal solution of the subproblem (\ref{dcgm-palm-subproblem}). Since each constraint involves only three variables, for the variables which are not involved in the constraints we set $y^{k}_{\bar{I}^{k}}=0$. The corresponding subproblem (\ref{dcgm-palm-subproblem}) can be replaced by solving the following problem
	\begin{eqnarray*}
		y^{k}_{I^{k}}\approx\mathop{\mbox{argmin}}_{y_{I^{k}}\in\mathcal{R}^{I^{k}}}\Big\{h(A_{S^{k},I^{k}}y_{I^{k}}-b_{S^{k}})+p(Dy_{I^{k}})\Big\},
	\end{eqnarray*}
	which can reduce the computational cost rapidly for each subproblem (\ref{dcgm-palm-subproblem}) if $|I^{k}|<<n_{1}$.
	
	\item Another point we need to mention is the computational cost of checking feasibility. In the implementation, we divide the constraints into different groups and check the feasibility in parallel.
	\item Due to Remark \ref{rem:palm} and the results in Section \ref{sec:CQ-nondegeneracy-invertible}, we set $H_{k}\equiv10^{-3}I$ for the $\ell_{1}$ and $\ell_{\infty}$ norm based problems and $H_{k}\equiv 10^{-6}I$ for the $\ell_{2}$ norm based problem in the numerical implementation.
	
\end{enumerate}

\section{Numerical experiments}
\label{sec:Numerical issues}
In this section, we implement some numerical experiments to demonstrate the efficiency of our DCGM$\_$PALM for the $\ell_{p}$ norm based metric nearness problems. All of the numerical experiments are implemented on a Windows workstation (two 16-core, Intel Xeon E5-2667 @3.20GHz CPU, 64GB RAM). In the implementation, all of these algorithms are written in C++ except the PROJECT AND FORGET algorithm\footnote{https://www.dropbox.com/sh/lq5nnhi4je2lh89/AABUUW7k5z3lXTSm8x1hhN1Da?dl=0} which is originally written in Julia.

At each iteration of DCGM$\_$PALM, we adopt the following relative KKT residual
\begin{eqnarray*}
	\eta_{kkt}&:=&\max\left\{\frac{\|A_{S}^{T}u_{S}+v\|}{1+\|v\|},\frac{\|A_{S}y-b_{S}-\Pi_{\mathcal{R}^{|S|}_{-}}(A_{S}y-b_{S}+u_{S})\|}{1+\|A_{S}y-b_{S}\|},\right.\\&&\quad\quad\quad\left.\frac{\|Dy-\mbox{Prox}_{q}(Dy+v)\|}{1+\|Dy\|}\right\}
\end{eqnarray*}
and the relative gap
\begin{eqnarray*}
	R_{g}:=\frac{|\mbox{pobj}-\mbox{dobj}|}{1+|\mbox{dobj}|}
\end{eqnarray*}
to measure the accuracy of PALM. It is terminated if the desired relative KKT residual $\eta_{kkt}<\mbox{tol}:=10^{-4}$ and the relative gap $R_{g}<\mbox{tol}$, or the number of the iterations of PALM reaches the maximum of 1000. For the consistency of the stopping criteria of these algorithms, we also adopt the constraint feasibility $\max\{A_{S}y-b_{S}\}<10^{-2}$ as the stopping criterion for PALM. The outer iteration of DCGM$\_$PALM is stopped if the the constraint feasibility of all the constraints $\eta_{f}:=\max\{Ay-b\}<10^{-2}$ and the relative KKT residual of the problem with the whole constraints is less than $\mbox{tol}$.

As for the Gurobi package, we also apply the same DCGM for the triangle inequalities constraints as that used in our algorithm to improve the efficiency of Gurobi (DCGM$\_$Gurobi), while the constraints related to the reformulations of the $\ell_{1}$ or $\ell_{\infty}$ norm are always in the constraint set. We use the barrier algorithm (without presolve and crossover) with the tolerances  $\mbox{BarConvTol}<\mbox{tol}$ and $\mbox{FeasibilityTol}<10^{-2}$ to solve each subproblem. The same stopping criterion as that of DCGM$\_$PALM is adopted for the outer iteration of DCGM.

For consistency, we make a minor modification of the stopping criterion in \cite{Veldt2019} and adopt
\begin{eqnarray*}
	\max\left\{R_{g},\ \eta_{f}\right\}<10^{-2}
\end{eqnarray*}
as the stopping criterion
for Dykstra's projection method. As the original code\footnote{https://github.com/nveldt/ParallelDykstras} for Dykstra's projection method is for the $\ell_{1}$ norm based problem, we rewrite it in C++ for the $\ell_{1}$, $\ell_{2}$, $\ell_{\infty}$ norms based problems and the same parallelization technology mentioned in \cite{Ruggles2020} is applied to improve the efficiency of Dykstra's projection method.
The parameters for the PROJECT AND FORGET algorithm are set as default.

In our experiments, we make the comparisons on some graphs with dissimilarity matrices derived from real-world networks which are from the SuiteSparse Matrix Collection \cite{Davis2011} and the SNAP repository \cite{Leskovec2014}. Before experiments, the edge weights are removed such that the networks are undirected. Then we find the largest connected component for further use. We follow the same approach in \cite{Veldt2019} to obtain the dissimilarity matrix $\widetilde{X}$ and the corresponding weight matrix $W$ with each element being nonzero.

In the following subsections, we report the name of data (Graph), the parameters $(n,n_{e})$ for the numbers of vertices and edges, the relative KKT residual $\eta_{kkt}$ (for DCGM$\_$PALM), the relative gap $R_{g}$ and the constraint feasibility $\eta_{f}$ for the problem with the whole constraints, the primal objective value (pobj), the iteration number (iter) and the computing time (time) in the format of hours:minutes:seconds.
\begin{table}[htbp]
\tiny
\parbox{\textwidth}{\caption{The performances of PALM, DCGM$\_$PALM, Gurobi, DCGM$\_$Gurobi for the  metric nearness problem. }\label{test_DCGM}}
\resizebox{\linewidth}{38mm}{
\centering
\begin{tabular}{lcc|cc|cc|cc}
\hline
\multicolumn{1}{l}{\multirow{1}{*}{Graph}} & \multicolumn{2}{c|}{PALM}                             & \multicolumn{2}{c|}{DCGM$\_$PALM}                        & \multicolumn{2}{c|}{Gurobi}            & \multicolumn{2}{c}{DCGM$\_$Gurobi}       \\ \cline{2-9}
\multicolumn{1}{l}{$(n,n_{e})$}                        & time                     & pobj                       & time(iter)                     & pobj                       & time              & pobj               & time(iter)              & pobj               \\ \hline
$\ell_{1}$ norm based problems\\\hline
jazz                                        & \multirow{2}{*}{0:04:48} & \multirow{2}{*}{2.46e+02} & \multirow{2}{*}{0:00:18(2)} & \multirow{2}{*}{2.46e+02} & \multirow{2}{*}{0:07:44} & \multirow{2}{*}{2.46e+02} & \multirow{2}{*}{0:01:13(2)} & \multirow{2}{*}{2.46e+02} \\
$(198,2742)$                                       &                          &                           &                          &                           &                   &                   &                   &                  \\
SmallW                                      & \multirow{2}{*}{0:09:48} & \multirow{2}{*}{7.88e+02} & \multirow{2}{*}{0:01:40(4)} & \multirow{2}{*}{7.88e+02} & \multirow{2}{*}{0:14:45} & \multirow{2}{*}{7.87e+02} & \multirow{2}{*}{0:06:02(4)} & \multirow{2}{*}{7.88e+02} \\
$(233,994)$                                       &                          &                           &                          &                           &                   &                   &                   &                   \\
celegansneural                              & \multirow{2}{*}{0:16:43} & \multirow{2}{*}{7.14e+02} & \multirow{2}{*}{0:01:47(3)} & \multirow{2}{*}{7.14e+02} & \multirow{2}{*}{0:54:33} & \multirow{2}{*}{7.14e+02} & \multirow{2}{*}{0:22:37(3)} & \multirow{2}{*}{7.14e+02} \\
$(297,2148)$                                        &                          &                           &                          &                           &                   &                   &                   &                  \\
USAir97                                     & \multirow{2}{*}{0:22:59} & \multirow{2}{*}{7.32e+02} & \multirow{2}{*}{0:01:05(3)} & \multirow{2}{*}{7.32e+02} & \multirow{2}{*}{0:43:34} & \multirow{2}{*}{7.32e+02} & \multirow{2}{*}{0:11:54(3)} & \multirow{2}{*}{7.32e+02} \\
$(332,2126)$                                      &                          &                           &                          &                           &                   &                   &                   &                  \\ \hline $\ell_{2}$ norm based problems\\\hline
jazz                                        & \multirow{2}{*}{0:04:35} & \multirow{2}{*}{1.19e+01} & \multirow{2}{*}{0:00:20(4)} & \multirow{2}{*}{1.19e+01} & \multirow{2}{*}{1:08:25} & \multirow{2}{*}{1.17e+01} & \multirow{2}{*}{8:38:29(4)} & \multirow{2}{*}{1.19e+01} \\
$(198,2742)$                                       &                          &                           &                          &                           &                   &                   &                   &                  \\
SmallW                                      & \multirow{2}{*}{0:07:48} & \multirow{2}{*}{1.74e+01} & \multirow{2}{*}{0:00:48(5)} & \multirow{2}{*}{1.74e+01} & \multicolumn{2}{c|}{\multirow{2}{*}{out of memory}} & \multicolumn{2}{c}{\multirow{2}{*}{out of memory}} \\
$(233,994)$                                       &                          &                           &                          &                           &                   &                   &                   &                   \\
celegansneural                              & \multirow{2}{*}{0:15:37} & \multirow{2}{*}{1.73e+01} & \multirow{2}{*}{0:00:42(2)} & \multirow{2}{*}{1.73e+01} & \multicolumn{2}{c|}{\multirow{2}{*}{out of memory}} & \multirow{2}{*}{17:39:14(2)} & \multirow{2}{*}{1.73e+01} \\
$(297,2148)$                                        &                          &                           &                          &                           &                   &                   &                   &                  \\
USAir97                                     & \multirow{2}{*}{0:22:09 } & \multirow{2}{*}{1.80e+01} & \multirow{2}{*}{0:00:33(2)} & \multirow{2}{*}{1.80e+01} & \multicolumn{2}{c|}{\multirow{2}{*}{out of memory}} & \multicolumn{2}{c}{\multirow{2}{*}{out of memory}} \\
$(332,2126)$                                      &                          &                           &                          &                           &                   &                   &                   &                  \\ \hline

$\ell_{\infty}$ norm based problems\\\hline
jazz                                        & \multirow{2}{*}{0:04:24} & \multirow{2}{*}{7.49e-02} & \multirow{2}{*}{0:00:26(4)} & \multirow{2}{*}{7.49e-02} & \multirow{2}{*}{0:05:30} & \multirow{2}{*}{7.45e-02} & \multirow{2}{*}{0:32:21(10)} & \multirow{2}{*}{7.50e-02} \\
$(198,2742)$                                       &                          &                           &                          &                           &                   &                   &                   &                  \\
SmallW                                      & \multirow{2}{*}{0:07:13} & \multirow{2}{*}{8.31e-02} & \multirow{2}{*}{0:00:38(5)} & \multirow{2}{*}{8.32e-02} & \multirow{2}{*}{0:08:41} & \multirow{2}{*}{8.29e-02} & \multirow{2}{*}{0:14:42(5)} & \multirow{2}{*}{8.31e-02} \\
$(233,994)$                                       &                          &                           &                          &                           &                   &                   &                   &                   \\
celegansneural                              & \multirow{2}{*}{0:15:36} & \multirow{2}{*}{7.60e-02} & \multirow{2}{*}{0:01:01(4)} & \multirow{2}{*}{7.60e-02} & \multirow{2}{*}{0:23:43} & \multirow{2}{*}{7.57e-02} & \multirow{2}{*}{3:17:53(17)} & \multirow{2}{*}{7.60e-02} \\
$(297,2148)$                                        &                          &                           &                          &                           &                   &                   &                   &                  \\
USAir97                                     & \multirow{2}{*}{0:21:50} & \multirow{2}{*}{8.31e-02} & \multirow{2}{*}{0:00:51(5)} & \multirow{2}{*}{8.32e-02} & \multirow{2}{*}{0:39:49} & \multirow{2}{*}{8.30e-02} & \multirow{2}{*}{3:02:28(11)} & \multirow{2}{*}{8.31e-02} \\
$(332,2126)$                                      &                          &                           &                          &                           &                   &                   &                   &                  \\ \hline

\end{tabular}
}
\end{table}

We first test the efficiency of DCGM. The performances of PALM, \\DCGM$\_$PALM, Gurobi, and DCGM$\_$Gurobi are listed in Table \ref{test_DCGM}. All of the  experiments in Table \ref{test_DCGM}  achieve the specified accuracy requirements unless being out of memory. From the comparison, for the $\ell_{1}$ norm based problem,  we see that the DCGM based algorithms are more efficient and we will only adopt the DCGM based algorithms for comparisons in the following numerical experiments. For the $\ell_{2}$ norm based problem, both Gurobi and DCGM$\_$Gurobi run out of memory when $n>300$. Therefore we only compare DCGM$\_$PALM with Dykstra's projection method in the numerical experiments of the $\ell_{2}$ norm based metric nearness problem.
For the $\ell_{\infty}$ norm based problem, the computing time of DCGM$\_$PALM is much less than that of PALM and we will only implement DCGM$\_$PALM in the numerical experiments. Although Gurobi runs better than DCGM$\_$Gurobi, we still adopt DCGM$\_$Gurobi for comparison due to the fact that the memory requirement of Gurobi grows quickly with the increasing of $n$.

\subsection{Numerical experiments for the $\ell_{1}$ norm based metric nearness problem}
In this section, we perform some numerical experiments of Dykstra's projection method, Gurobi and our proposed algorithm for the $\ell_{1}$ norm based metric nearness problem. It is known from \cite{Veldt2019} that the linear programming relaxation for
correlation clustering is equivalent to the $\ell_{1}$ norm based metric nearness problem.

\subsubsection{Solving the $\ell_{1}$ norm based metric nearness problem (\ref{primal-problem}) by linear programming}
The $\ell_{1}$ norm based metric nearness problem is a polyhedral problem and can be rewritten as a linear programming problem. By introducing some slack variables, we can reformulate problem (\ref{primal-problem}) to the following form.
\begin{eqnarray}\label{primal-problem-lp}
	&&\min_{z\in\mathcal{R}^{2n_{1}}}\left\{\langle c,z\rangle\ \left|\ \widetilde{A}z\leq\widetilde{b}\right.\right\},
\end{eqnarray}
where $z=[y;\tilde{y}]$, $c=[\textbf{0}_{n_{1}};\mbox{trivec}(W)]$, $\textbf{0}_{n_{1}}$ denotes an $n_{1}$-dimensional zero column vector and
\begin{eqnarray*}
	\widetilde{A}=\left[\begin{array}{cc}A&0\\I_{n_{1}}&-I_{n_{1}}\\-I_{n_{1}}&-I_{n_{1}}\end{array}\right],\quad \widetilde{b}=\left[\begin{array}{l}b\\\textbf{0}_{n_{1}}\\\textbf{0}_{n_{1}}\end{array}\right].
\end{eqnarray*}
The KKT condition of problem (\ref{primal-problem-lp}) takes the following form
\begin{eqnarray*}
	\widetilde{A}^{T}u+c=0,\quad u\in\mathcal{N}_{\mathcal{R}^{2n_{1}+n_{2}}_{-}}(\widetilde{A}z-\widetilde{b}).
\end{eqnarray*}


As we know, there exist many optimization software packages, e.g., Gurobi, which is an ideal solver to deal with the above linear programming problem. Hence, we use the Gurobi package for comparison.

\subsubsection{Dykstra's projection method}
Dykstra's projection methods cannot be applied directly to solve linear programming problems and we need to add an appropriate regularization term to the objective function of the linear programming problem (\ref{primal-problem-lp}). Consider the following quadric programming problem
\begin{eqnarray}\label{projection-method-l1}
\min_{z\in\mathcal{R}^{2n_{1}}}\left\{\langle c,z\rangle+\frac{1}{2\gamma}z^{T}\widetilde{D}z\ \left|\ \widetilde{A}z\leq\widetilde{b}\right.\right\},
\end{eqnarray}
where $\gamma>0,\,\widetilde{D}=\mbox{diag}([\mbox{trivec}(W);\mbox{trivec}(W)])$.
The dual problem of (\ref{projection-method-l1}) is
\begin{eqnarray*}
\min_{u\in\mathcal{R}^{2n_{1}+n_{2}}_{+}} \Big\{\frac{\gamma}{2}(\widetilde{A}^{T}u+c)^{T}\widetilde{D}^{-1}(\widetilde{A}^{T}u+c)+\langle\widetilde{b},u\rangle\Big\}
\end{eqnarray*}
and the corresponding KKT system takes the following form
\begin{eqnarray*}
\frac{1}{\gamma}\widetilde{D}z+\widetilde{A}^{T}u+c=0,\ \widetilde{A}z-\widetilde{b}\leq 0,\ u\geq 0,\ \langle u,\widetilde{A}z-\widetilde{b}\rangle=0.
\end{eqnarray*}

As proved in \cite{Mangasarian1984}, there exists $\gamma_{0}>0$ such that for all $\gamma>\gamma_{0}$ the optimal solution of the quadratic programming problem (\ref{projection-method-l1}) is the minimum $\ell_{2}$ norm solution of the original linear programming problem (\ref{primal-problem-lp}) when $\widetilde{D}=I_{2n_{1}}$. The function $f(z)=\delta_{\mathcal{R}^{2n_{1}}_{-}}(\widetilde{A}z-\widetilde{b})+\langle c,z\rangle$ is polyhedral. For a convex closed polyhedral function $f$, its subdifferential has an interesting property called the staircase property (see e.g., Section 6 of \cite{Eckstein1992}), i.e, there exists $\delta>0$ such that
\begin{eqnarray*}
t\in\partial f(z),\ \|t\|\leq\delta\ \Rightarrow\ 0\in\partial f(z).
\end{eqnarray*}
Let $\bar{z}$ be an optimal solution of problem (\ref{projection-method-l1}), then we have
$\frac{1}{\gamma}\widetilde{D}\bar{z}\in\partial f(\bar{z})$. It is known that both solution sets of (\ref{primal-problem-lp}) and (\ref{projection-method-l1}) are bounded.  Therefore for a general weight matrix, there also exists $\gamma_{1}>0$ such that for all $\gamma>\gamma_{1}$ the solution of the quadratic programming problem (\ref{projection-method-l1}) is also a solution of the original problem.

Due to the difficulty of determining the parameter $\gamma_{1}$ exactly and the unknown dual variables $u$ and $v$ such that the KKT residual of the original problem is undetermined, we set $\gamma=1$, which is the same as that in \cite{Veldt2019}, in the numerical experiments.
\begin{landscape}
	
\begin{table}[htbp]
\tiny
\centering
\parbox{1.4\textwidth}{\caption{The performances of DCGM$\_$PALM, DCGM$\_$Gurobi, Dykstra's projection method and PROJECT AND FORGET for the $\ell_{1}$ norm based metric nearness problem.}\label{l1-norm-results}}
\resizebox{1.5\textwidth}{17mm}{
\begin{tabular}{|l|ccccc|ccccc|cccc|cc|}
\hline
\multicolumn{1}{|l}{Graph} & \multicolumn{5}{|c}{DCGM$\_$PALM} & \multicolumn{5}{|c}{DCGM$\_$Gurobi} & \multicolumn{4}{|c}{Dykstra's projection method}& \multicolumn{2}{|c|}{PROJECT AND FORGET} \\ \cline{2-17}
\multicolumn{1}{|l|}{($n$,$n_{e}$)} & $\eta_{kkt}$ & $R_{g}$ & $\eta_{f}$ & pobj  & time(iter) & $\eta_{kkt}$ & $R_{g}$ & $\eta_{f}$ & pobj  & time(iter) & $R_{g}$ & $\eta_{f}$ & pobj & time(iter)             & pobj & time(iter)  \\ \hline
caGrQc & \multirow{2}{*}{6.39e-05} & \multirow{2}{*}{2.74e-06} & \multirow{2}{*}{4.38e-03} & \multirow{2}{*}{2.88e+03}         &\multirow{2}{*}{0:12:38(4)} &   \multicolumn{5}{c|}{\multirow{2}{*}{out of memory}}  &\multirow{2}{*}{7.23e-05}  &\multirow{2}{*}{9.80e-03} &\multirow{2}{*}{3.07e+03} & \multirow{2}{*}{1:13:15(270)} &  \multirow{2}{*}{4.04+03} & \multirow{2}{*}{1:28:23(99)}\\
(4158,13422) & & & & & & & & & & & & & & & &  \\\hline
power & \multirow{2}{*}{9.05e-05} &\multirow{2}{*}{1.74e-07} & \multirow{2}{*}{2.38e-03} & \multirow{2}{*}{2.01e+03}  & \multirow{2}{*}{0:03:29(5)} &\multirow{2}{*}{2.57e-07} &\multirow{2}{*}{5.54e-05} &\multirow{2}{*}{1.12e-04} & \multirow{2}{*}{2.01e+03} & \multirow{2}{*}{0:05:16(4)} &\multirow{2}{*}{4.62e-05}  &\multirow{2}{*}{8.62e-03} & \multirow{2}{*}{2.15e+03} & \multirow{2}{*}{1:03:24(150)} & \multirow{2}{*}{2.59+03} & \multirow{2}{*}{0:39:30(26)}  \\
(4941,6594) & & & & & & & & & & & & & & & & \\\hline
caHepTh & \multirow{2}{*}{8.80e-05} & \multirow{2}{*}{2.44e-06} & \multirow{2}{*}{9.24e-03} & \multirow{2}{*}{7.05e+03}
& \multirow{2}{*}{1:12:05(4)}  & \multicolumn{5}{c|}{\multirow{2}{*}{out of memory}} &\multirow{2}{*}{6.05e-05}  &\multirow{2}{*}{9.26e-03} & \multirow{2}{*}{7.39e+03} & \multirow{2}{*}{14:25:46(310)}      & \multicolumn{2}{c|}{\multirow{2}{*}{out of memory}}   \\
(8638,24806) & & & & & & & & & & & & & & & & \\\hline
caHepPh & \multirow{2}{*}{6.51e-05} & \multirow{2}{*}{2.10e-06} & \multirow{2}{*}{5.60e-03}  & \multirow{2}{*}{1.39e+04}  & \multirow{2}{*}{7:21:37(6)} & \multicolumn{5}{c|}{\multirow{2}{*}{out of memory}}  &\multirow{2}{*}{8.10e-05}  &\multirow{2}{*}{9.80e-03} & \multirow{2}{*}{1.39e+04} & \multirow{2}{*}{34:26:56(320)}     & \multicolumn{2}{c|}{\multirow{2}{*}{out of memory}}  \\
(11204,117619) & & & & & & & & & & & & & & & & \\\hline
caAstroPh & \multirow{2}{*}{3.40e-05} & \multirow{2}{*}{3.24e-06} & \multirow{2}{*}{8.31e-03} & \multirow{2}{*}{2.93e+04} & \multirow{2}{*}{19:56:56(4)} & \multicolumn{5}{c|}{\multirow{2}{*}{out of memory}} &\multirow{2}{*}{8.08e-05}  &\multirow{2}{*}{8.44e-03} & \multirow{2}{*}{3.11e+04} & \multirow{2}{*}{180:15:39(370)} & \multicolumn{2}{c|}{\multirow{2}{*}{out of memory}} \\
(17903,196972)& & & & & & & & & & & & & & & & \\\hline
\end{tabular}
}
\end{table}

\begin{table}[htbp]
\tiny
\centering
\parbox{1.1\textwidth}{\caption{The performances of DCGM$\_$PALM and Dykstra's projection method for the $\ell_{2}$ norm based metric nearness problem.}\label{l2-norm-results}}
\begin{tabular}{|l|ccccc|cccc|}
\hline
\multicolumn{1}{|l}{Graph} & \multicolumn{5}{|c}{DCGM$\_$PALM} & \multicolumn{4}{|c|}{Dykstra's projection method} \\ \cline{2-10}
\multicolumn{1}{|l|}{($n$,$n_{e}$)} & $\eta_{kkt}$ & $R_{g}$ & $\eta_{f}$ & pobj  & time(iter) & $R_{g}$ & $\eta_{f}$ & pobj & time(iter)             \\ \hline
caGrQc &\multirow{2}{*}{9.49e-05}  &\multirow{2}{*}{1.04e-06}  &\multirow{2}{*}{2.95e-03}         & \multirow{2}{*}{3.65e+01}         &\multirow{2}{*}{0:02:40(2)}  &\multirow{2}{*}{1.99e-04}  &\multirow{2}{*}{9.78e-03}  & \multirow{2}{*}{3.65e+01}  &  \multirow{2}{*}{0:14:43(60)} \\
(4158,13422) & & &  & & & & & &  \\\hline
power &\multirow{2}{*}{5.58e-05} &\multirow{2}{*}{6.99e-07} &\multirow{2}{*}{7.70e-03} & \multirow{2}{*}{2.92e+01} & \multirow{2}{*}{0:01:24(2)}  &\multirow{2}{*}{5.99e-04} &\multirow{2}{*}{5.64e-03} & \multirow{2}{*}{2.92e+01} & \multirow{2}{*}{0:16:46(40)}  \\
(4941,6594) & & &  & & & & & & \\\hline
caHepTh &\multirow{2}{*}{8.30e-05} &\multirow{2}{*}{1.04e-06} &\multirow{2}{*}{2.85e-03} & \multirow{2}{*}{5.68e+01}
& \multirow{2}{*}{0:11:09(2)}   &\multirow{2}{*}{5.72e-05} &\multirow{2}{*}{9.60e-03} & \multirow{2}{*}{5.68e+01} & \multirow{2}{*}{2:48:32(60)}        \\
(8638,24806) & & &  & & & & & & \\\hline
caHepPh &\multirow{2}{*}{6.83e-05} &\multirow{2}{*}{3.73e-06} &\multirow{2}{*}{2.20e-03} & \multirow{2}{*}{8.35e+01}  & \multirow{2}{*}{0:48:02(3)}   &\multirow{2}{*}{6.94e-04} &\multirow{2}{*}{6.90e-03} & \multirow{2}{*}{8.35e+01} & \multirow{2}{*}{8:12:45(70)}     \\
(11204,117619) & &  & & & & & & & \\\hline
caAstroPh &\multirow{2}{*}{8.34e-05} &\multirow{2}{*}{5.62e-07} &\multirow{2}{*}{4.91e-03} & \multirow{2}{*}{1.20e+02} & \multirow{2}{*}{2:00:26(3)}  &\multirow{2}{*}{5.66e-05} &\multirow{2}{*}{8.92e-03} & \multirow{2}{*}{1.20e+02} & \multirow{2}{*}{66:43:53(120)} \\
(17903,196972)& &  & & & & & & & \\\hline
\end{tabular}
\end{table}

\end{landscape}
\subsubsection{Numerical results for the $\ell_{1}$ norm based metric nearness problem}
In this subsection, we make the comparisons of DCGM$\_$PALM, DCGM$\_$Gurobi, Dykstra's projection method and PROJECT AND FORGET. It can be shown from Table \ref{l1-norm-results} that although DCGM$\_$PALM and Dykstra's projection method can obtain the desired results with the corresponding  stopping criteria, The objective value obtained by DCGM$\_$PALM is less than that obtained by Dykstra's projection method for almost all the data sets. The basic reason lies in the extra proximal term. Although the cost of checking feasibility of the whole constraints is $O(n^{3})$, the number of the constraint generation iterations of DCGM$\_$PALM is much less than that of Dykstra's projection method.

Furthermore, we can solve each subproblem of DCGM$\_$PALM efficiently due to the structure of the problem. Since there are only 3 nonzero elements in each row of the constraint matrix, the number of the variables involved in the constraint set $S$ is less than $n_{1}$ in each subproblem. We take the data set caHepPh as an example. The sizes of the constraint sets of the subproblems are 18877829, 37755658, 56648464, 84998076, 169656800 and 171628352, respectively. While the numbers of the variables involved in the corresponding subproblems are 3304117, 3600822, 4105644, 4896376, 6164352 and 6208620, respectively, which are much less than $n_{1}=62759206$.


In Table \ref{l1-norm-results}, we can see that DCGM$\_$Gurobi can only solve the second data problem, while PROJECT AND FORGET can only solve the first two data problems. For other data problems, both solvers run out of memory. Although Dykstra's projection method can solve all the data problems, its computing time is much more than that of DCGM$\_$PALM.

We also test some graphs with the number of nodes larger than $2\times 10^{4}$ and the number of the constraints up to $10^{13}$ . The number of the total constraints (no$\_$cons), the largest size of the constraint set (active$\_$cons) of DCGM$\_$PALM, the computing time of DCGM and checking feasibility of the constraints (DCGM+FEAS), the computing time of solving the subproblems (PALM) are also listed in Table \ref{l1-norm-large-graph}. The graphs are sparse and the size of the constraint set in DCGM is much less than  $n_{2}$.  Due to the sparse restoration, the maximum memory requirement for the $\ell_{1}$ norm based problems related to these data sets is less than 45G.

\begin{table}[htbp]
\tiny
\centering
\parbox{0.9\textwidth}{\caption{The performances of DCGM$\_$PALM for the $\ell_{1}$ norm based metric nearness problem of large data sets.}\label{l1-norm-large-graph}}
\resizebox{\textwidth}{12mm}{
\begin{tabular}{lcccccccr}
\hline
Graph &\multirow{2}{*}{no$\_$cons}&\multirow{2}{*}{ active$\_$cons}& \multirow{2}{*}{iter}& \multirow{2}{*}{$\eta_{kkt}$} & \multirow{2}{*}{$R_{g}$} & \multirow{2}{*}{$\eta_{f}$} & \multirow{2}{*}{pobj} & \multirow{2}{*}{time(DCGM+FEAS $|$ PALM)} \\
($n$,$n_{e}$) & & & & & & & &\\ \hline
CA-CondMat&\multirow{2}{*}{4.9e+12}&\multirow{2}{*}{21303932} &\multirow{2}{*}{5} &\multirow{2}{*}{3.21e-05}  &\multirow{2}{*}{3.78e-06}  & \multirow{2}{*}{9.19e-03}  & \multirow{2}{*}{2.48e+04}         &\multirow{2}{*}{14:28:09(3:02:52 $|$ 11:25:17)}  \\
(21363,91342) & & & & & & & & \\\hline
p2p-Gnutella25&\multirow{2}{*}{5.8e+12}&\multirow{2}{*}{8140970} &\multirow{2}{*}{4} &\multirow{2}{*}{1.06e-05} &\multirow{2}{*}{9.74e-07} &\multirow{2}{*}{2.48e-03} & \multirow{2}{*}{1.98e+04} & \multirow{2}{*}{6:08:29(3:02:36 $|$ 3:05:53)}  \\
(22663,54693) & & & & & & & &\\\hline
p2p-Gnutella30&\multirow{2}{*}{2.5e+13}&\multirow{2}{*}{13198004} &\multirow{2}{*}{4} &\multirow{2}{*}{5.62e-05} &\multirow{2}{*}{6.36e-07} &\multirow{2}{*}{7.93e-03} & \multirow{2}{*}{3.15e+04} & \multirow{2}{*}{22:19:22(12:59:35 $|$ 9:19:47)}  \\
(36646,88303) & & & & & & & & \\\hline
\end{tabular}
}
\end{table}

\subsection{Numerical experiments for the $\ell_{2}$ norm based metric nearness problem}

\begin{table}[htbp]
\tiny
\centering
\parbox{.9\textwidth}{\caption{The performances of DCGM$\_$PALM for the $\ell_{2}$ norm based metric nearness problem of large data sets.}\label{l2-norm-large-graph}}
\resizebox{\textwidth}{12mm}{
\begin{tabular}{lcccccccr}
\hline
Graph &\multirow{2}{*}{no$\_$cons}&\multirow{2}{*}{ active$\_$cons}& \multirow{2}{*}{iter}& \multirow{2}{*}{$\eta_{kkt}$} & \multirow{2}{*}{$R_{g}$} & \multirow{2}{*}{$\eta_{f}$} & \multirow{2}{*}{pobj} & \multirow{2}{*}{time(DCGM+FEAS $|$ PALM)} \\
($n$,$n_{e}$) & & & & & & & &\\ \hline
CA-CondMat&\multirow{2}{*}{4.9e+12}&\multirow{2}{*}{185711148} &\multirow{2}{*}{3} &\multirow{2}{*}{7.44e-05}  &\multirow{2}{*}{3.58e-07}  & \multirow{2}{*}{3.33e-03}        & \multirow{2}{*}{1.08e+02}         &\multirow{2}{*}{2:31:08(2:06:13 $|$ 0:24:55)}  \\
(21363,91342) & & & & & & & & \\\hline
p2p-Gnutella25&\multirow{2}{*}{5.8e+12}&\multirow{2}{*}{7736507} &\multirow{2}{*}{2} &\multirow{2}{*}{9.45e-05} &\multirow{2}{*}{1.31e-06} &\multirow{2}{*}{4.69e-03} & \multirow{2}{*}{9.44e+01} & \multirow{2}{*}{2:15:29(1:54:12 $|$ 0:21:17)}  \\
(22663,54693) & & & & & & & & \\\hline
p2p-Gnutella30&\multirow{2}{*}{2.5e+13}&\multirow{2}{*}{12427811} &\multirow{2}{*}{2} &\multirow{2}{*}{9.72e-05} &\multirow{2}{*}{5.16e-07} &\multirow{2}{*}{7.93e-03} & \multirow{2}{*}{1.19e+02} & \multirow{2}{*}{8:32:33(7:56:50 $|$ 0:35:43)}  \\
(36646,88303) & & & & & & & &\\\hline
\end{tabular}
}
\end{table}

In this section, we perform some numerical experiments for the $\ell_{2}$ norm based metric nearness problem. Since Dykstra's projection method and the Gurobi package can only be applied to solve the quadratic programming problem, we reformulate the $\ell_{2}$ norm based metric nearness problem equivalently to the problem with the objective function being the square of the $\ell_{2}$ norm. We have tested DCGM$\_$PALM on several data sets. The computing time of DCGM$\_$PALM for the original problem is less than that for the squared $\ell_{2}$ norm problem. The reason may lie in that the minimal value of the weight matrix is very small, which may lead to ill-conditioning of the inner subproblem. Therefore, we only implement the numerical experiments of DCGM$\_$PALM for the original problem in the following implementation. The corresponding objective values of the original problem are listed in Tables \ref{l2-norm-results} and \ref{l2-norm-large-graph}.

As for the Gurobi package,
we also apply DCGM and adopt the same stopping criterion as that of DCGM$\_$PALM. As shown in Table \ref{test_DCGM}, both Gurobi and DCGM$\_$Gurobi
run out of memory for the data sets with $n>200$. We only compare our DCGM$\_$PALM with Dykstra's projection method for the $\ell_{2}$ norm based metric nearness problem. 

Both of the algorithms can solve the $\ell_{2}$ norm based metric nearness problem with the given corresponding stopping criterion,
however, the results in Table \ref{l2-norm-results} demonstrate that our DCGM$\_$PALM is much more efficient. The number of iterations for DCGM is small (less than 5), which means we only need to take a few  iterations with each iteration $O(n^{3})$ cost to check the feasibility of the constraints.

The performance results of DCGM$\_$PALM for the $\ell_{2}$ norm based metric nearness problem with other large data sets are listed in Table \ref{l2-norm-large-graph}.
The computing time of solving the subproblems is much less than that of the delayed constraint generation and checking feasibility of constraints reveals that PALM can solve each subproblem efficiently. The high efficiency is due to DCGM and the SsN based PALM. We take the data set p2p-Gnutella30 as an example. The numbers of the constraints for the two subproblems are 12347944 and 12427811, respectively, while the total number of the constraints is $24604479499740$. The numbers of the variables involved in the corresponding subproblems are
11169993 and 11186321, respectively, while the total number of the variables is $671446335$. The maximum memory requirement is less than 30G.

\subsection{Numerical experiments for the $\ell_{\infty}$ norm based metric nearness problem}

In this section, we perform some numerical experiments of Dykstra's projection method, DCGM$\_$Gurobi and our proposed algorithm for the $\ell_{\infty}$ norm based metric nearness problem.

We first reformulate the $\ell_{\infty}$ norm based metric nearness problem as a linear programming problem. By introducing some slack variables, we can write out the problem (\ref{primal-problem}) as the following form.
\begin{eqnarray}\label{primal-problem-lp-infty}
\min_{z\in\mathcal{R}^{n_{1}+1}}\left\{\langle \widehat{c},z\rangle\ \left|\ \widehat{A}z\leq\widehat{b}\right.\right\},
\end{eqnarray}
where $\widehat{c}=[\textbf{0}_{n_{1}};1]$ and

\begin{align*}
\widehat{A}=\left[\begin{array}{cc}A&\textbf{0}_{n_{2}}\\I_{n_{1}}&-\textbf{1}_{n_{1}}\\-I_{n_{1}}&-\textbf{1}_{n_{1}}\end{array}\right],\quad \widehat{b}=\left[\begin{array}{l}b\\\textbf{0}_{n_{1}}\\\textbf{0}_{n_{1}}\end{array}\right],
\end{align*}
where $\textbf{1}_{n_{1}}$ is an $n_{1}$-dimensional column vector with all elements being ones. When we apply Dykstra's projection method to solve problem (\ref{primal-problem-lp-infty}), we need to approximate problem (\ref{primal-problem-lp-infty}) by the following quadratic programming problem
\begin{eqnarray}\label{projection-method-infty}
\min_{z\in\mathcal{R}^{n_{1}+1}}\left\{\langle \widehat{c},z\rangle+\frac{1}{2\gamma}z^{T}\widehat{D}z\ \left|\ \widehat{A}z\leq\widehat{b}\right.\right\},
\end{eqnarray}
where $\widehat{D}=\mbox{diag}([\mbox{trivec}(W);1])$.

\begin{table}[t]
\parbox{\textwidth}{\caption{The performances of Dykstra's projection method for the $\ell_{\infty} $ norm based metric nearness problem with different $\gamma$. }\label{project_small}}
\centering
\begin{tabular}{lcccccr}
\hline
Graph                       & Gurobi$\_$obj                & $\gamma$ & obj      & ratio    & iter & time \\ \hline
 & \multirow{5}{*}{8.23e-02} & 1     & 7.00e-01 & 849.74\% & 55        & 0:00:09       \\
   Harvard500      &              & 10    & 5.77e-01 & 700.62\% & 115       & 0:00:19       \\
    $n$ = 500         &          & 100   & 3.97e-01 & 482.19\% & 1555      & 0:04:17     \\
     $n_{e}$ = 2043   &           & 200   & 3.20e-01 & 388.40\% & 3175      & 0:08:10     \\
                &           & 500   & 2.60e-01 & 316.23\% & 10560     & 0:25:39 \\\hline
& \multirow{5}{*}{8.29e-02}     & 1     & 3.02e-01 &  364.62\%        & 40        & 0:00:43       \\
email   &               & 10    & 2.59e-01 &  312.34\%        & 165       & 0:02:54      \\
$n$ = 1133   &          & 100   & 2.02e-01 & 244.12\%         & 1460      & 0:25:37     \\
$n_{e}$ = 5451         &            & 200   & 1.81e-01 &  217.06\%        & 3320      & 0:58:23     \\
          &                & 500   & 1.48e-01 &  178.80\%        & 10330     & 2:55:34 \\ \hline
\end{tabular}
\end{table}

In order to set an appropriate parameter $\gamma$ such that the solution of the quadratic programming problem (\ref{projection-method-infty}) is close enough to the original problem (\ref{primal-problem-lp-infty}), we test several different parameters $\gamma$ for some small scale data sets. In Table \ref{project_small}, in addition to some indices introduced previously, we also list the objective value (Gurobi$\_$obj), which is obtained by the Gurobi package with no constraint generation strategy, the value of the parameter $\gamma$ ($\gamma$), and the ratio of the objective value obtained by Dykstra's projection method and that obtained by Gurobi (ratio). We can see that the parameter $\gamma$ should be relatively larger (be greater than 500) such that the original objective value obtained by Dykstra's projection method is relatively closer to the true objective value. However the number of the iterations grows rapidly as $\gamma$ increases and so does the computing time. For the trade-off, we set $\gamma=500$ in the following comparison.

We present the performance results of DCGM$\_$PALM, DCGM$\_$Gurobi and Dykstra's projection method for the $\ell_{\infty}$ norm based metric nearness problem in Table \ref{linfty-norm-results}.
As is seen from Table \ref{project_small}, for Dykstra's projection method, the number of the iterations is very large and it takes a lot of time.
Hence in Table \ref{linfty-norm-results},
We only list the computing time of one iteration (time$\_$per$\_$iter) and the estimated time (estimated$\_$time) for Dykstra's projection method. We assume that the number of the iterations for Dykstra's projection method is $10^{4}$, though the actual number may be larger than $10^{4}$.
We observe from Table \ref{linfty-norm-results} that DCGM$\_$Gurobi cannot obtain the desired results for these data sets when $n>4000$ due to the memory limitation. The computing time of Dykstra's projection method grows rapidly with the increasing of $n$ and it becomes unacceptable when $n$ is greater than $10^{4}$. DCGM$\_$PALM can obtain the desired approximate solution with the given accuracy for these data sets and the computing time is much less than that of Dykstra's projection method.

In Table \ref{linfty-norm-large-graph}, we only present the performance results of DCGM$\_$PALM for the $\ell_{\infty}$ norm based metric nearness problem of large data sets. The results in this table demonstrates that our DCGM$\_$PALM can solve the $\ell_{\infty}$ norm based metric nearness problem with $n$ up to $3\times10^{4}$ and the number of the constraints up to $10^{13}$. Since the size of the constraint set is much less than the total number of the constraints and the number of the variables involved in the constraint set is much less than the total number of the variables, the SsN based PALM solves the corresponding subproblems efficiently and the total time cost of DCGM$\_$PALM is satisfactory. Take the data set p2p-Gnutella30 as an example. The numbers of the constraints for these subproblems are 12347944, 24695888, 49391776, 98783552, 104703160 and 104860475, respectively, while the total number of the constraints is $24604479499740$. The numbers of the variables involved in the corresponding subproblems are
11169993, 15163073, 18173253, 23562258, 24825734 and 24829265, respectively, while the total number of the variables is $671446335$. The maximum memory requirement is less than 45G.


\section{Conclusion}
\label{sec:Conclusion}
In this paper, we have introduced an efficient DCGM$\_$PALM to solve the $\ell_{p}$ $(p=1,2,\infty)$ norm based metric nearness problem.
An efficient SsN based PALM is applied to solve each subproblem of DCGM. We take full advantage of the special structure of the corresponding problem and overcome the storage difficulty such that we can solve the metric nearness problem with $n$ up to $3\times 10^{4}$, especially for sparse graphs. DCGM$\_$PALM can solve the $\ell_{p}\ (p=1,2,\infty)$ norm based metric nearness problem efficiently and the memory requirement is acceptable. Numerical experiments on several real graph demonstrate the efficiency of our algorithm. As far as we know, it is the first time that numerical implementations on the $\ell_{\infty}$ norm based metric nearness problem can be conducted with $n$ greater than $10^{5}$.

In theory, we have established the primal-dual error bound condition for PALM. We have also established
the equivalence between the dual nondegeneracy condition and nonsingularity of the generalized Jacobian for the inner subproblem of PALM. Furthermore, when $q(\cdot)=\|\cdot\|_{1}$ or $\|\cdot\|_{\infty}$, we have established the equivalence between the SRCQ and the nondegeneracy condition, and the equivalence between the dual nondegeneracy condition and the uniqueness of the primal solution.

\section*{Acknowledgements}
\begin{landscape}
	
\begin{table}[htbp]
\tiny
\centering
\parbox{1.1\textwidth}{\caption{The performances of DCGM$\_$PALM, DCGM$\_$Gurobi and Dykstra's projection method for the $\ell_{\infty}$ norm based metric nearness problem.}\label{linfty-norm-results}}
\begin{tabular}{|l|ccccc|ccccc|cc|}
\hline
\multicolumn{1}{|l}{Graph} & \multicolumn{5}{|c}{DCGM$\_$PALM} & \multicolumn{5}{|c}{DCGM$\_$Gurobi} & \multicolumn{2}{|c|}{Dykstra's projection method}\\ \cline{2-13}
\multicolumn{1}{|l|}{($n$,$n_{e}$)} & $\eta_{kkt}$ & $R_{g}$ & $\eta_{f}$ & pobj  & time(iter) & $\eta_{kkt}$ & $R_{g}$ & $\eta_{f}$ & pobj  & time(iter) & time$\_$per$\_$iter & estimated$\_$time \\ \hline
caGrQc &\multirow{2}{*}{7.79e-05}  & \multirow{2}{*}{7.68e-06} &\multirow{2}{*}{6.74e-03}         & \multirow{2}{*}{8.31e-02}         &\multirow{2}{*}{0:11:16(4)} & \multicolumn{5}{c|}{\multirow{2}{*}{out of memory}} & \multirow{2}{*}{0:00:16}  &  \multirow{2}{*}{45h}\\
(4158,13422) & & & & & & & & & & & & \\\hline
power &\multirow{2}{*}{3.16e-05} &\multirow{2}{*}{3.10e-06} &\multirow{2}{*}{1.68e-03} & \multirow{2}{*}{8.31e-02} & \multirow{2}{*}{0:04:37(6)} & \multicolumn{5}{c|}{\multirow{2}{*}{out of memory}} & \multirow{2}{*}{0:00:25} & \multirow{2}{*}{70h}\\
(4941,6594) & & & & & & & & & & & & \\\hline
caHepTh &\multirow{2}{*}{7.75e-05} &\multirow{2}{*}{7.08e-05} &\multirow{2}{*}{4.63e-03} & \multirow{2}{*}{8.32e-02}
& \multirow{2}{*}{0:52:27(4)}  & \multicolumn{5}{c|}{\multirow{2}{*}{out of memory}} & \multirow{2}{*}{0:02:39} & \multirow{2}{*}{441h}      \\
(8638,24806) & & & & & & & & & & & & \\\hline
caHepPh &\multirow{2}{*}{1.81e-05} & \multirow{2}{*}{6.62e-06}& \multirow{2}{*}{6.10e-03} & \multirow{2}{*}{8.31e-02}  & \multirow{2}{*}{11:13:40(8)} & \multicolumn{5}{c|}{\multirow{2}{*}{out of memory}}  & \multirow{2}{*}{0:06:23} & \multirow{2}{*}{1063h} \\
(11204,117619) & & & & & & & & & & & & \\\hline
caAstroPh &\multirow{2}{*}{6.34e-05} &\multirow{2}{*}{5.15e-06} &\multirow{2}{*}{5.92e-03} & \multirow{2}{*}{8.31e-02} & \multirow{2}{*}{63:11:53(7)} & \multicolumn{5}{c|}{\multirow{2}{*}{out of memory}} & \multirow{2}{*}{0:29:14} & \multirow{2}{*}{4872h}\\
(17903,196972)& & & & & & & & & & & & \\\hline
\end{tabular}
\end{table}

\begin{table}[htbp]
\tiny
\centering
\parbox{.9\textwidth}{\caption{The performances of DCGM$\_$PALM for the $\ell_{\infty}$ norm based metric nearness problem of large data sets.}\label{linfty-norm-large-graph}}
\begin{tabular}{lcccccccr}
\hline
Graph &\multirow{2}{*}{no$\_$cons}&\multirow{2}{*}{ active$\_$cons}& \multirow{2}{*}{iter}& \multirow{2}{*}{$\eta_{kkt}$} & \multirow{2}{*}{$R_{g}$} & \multirow{2}{*}{$\eta_{f}$} & \multirow{2}{*}{pobj} & \multirow{2}{*}{time(DCGM+FEAS $|$ PALM)} \\
($n$,$n_{e}$) & & & & & & & &\\ \hline
CA-CondMat&\multirow{2}{*}{4.9e+12} &\multirow{2}{*}{52044885 }&\multirow{2}{*}{4} &\multirow{2}{*}{3.87e-05 }  &\multirow{2}{*}{8.34e-05 }  & \multirow{2}{*}{8.08e-03 }        & \multirow{2}{*}{8.32e-02 }         &\multirow{2}{*}{14:10:41(2:36:41 $|$ 11:34:00) }  \\
(21363,91342) & & & & & & & & \\\hline
p2p-Gnutella25&\multirow{2}{*}{5.8e+12}&\multirow{2}{*}{67528494} &\multirow{2}{*}{6} &\multirow{2}{*}{3.66e-05} &\multirow{2}{*}{7.47e-05 } &\multirow{2}{*}{2.07e-03} & \multirow{2}{*}{8.32e-02} & \multirow{2}{*}{17:07:44(4:29:54 $|$ 12:37:50) }  \\
(22663,54693) & & & & & & & & \\\hline
p2p-Gnutella30&\multirow{2}{*}{2.5e+13}&\multirow{2}{*}{104860475} &\multirow{2}{*}{6} &\multirow{2}{*}{1.57e-05} &\multirow{2}{*}{1.55e-05} &\multirow{2}{*}{2.44e-03} & \multirow{2}{*}{8.31e-02} & \multirow{2}{*}{43:56:41(18:25:19 $|$ 25:31:22) }  \\
(36646,88303) & & & & & & & & \\\hline
\end{tabular}
\end{table}

\end{landscape}

\bibliographystyle{amsplain}
\bibliography{ref}

\end{document}